\documentclass[oneside]{amsart}

\usepackage{amsthm}
\usepackage{amsmath}
\usepackage{amssymb}
\usepackage{enumerate}
\usepackage{graphicx}
\usepackage[hidelinks,pagebackref,pdftex]{hyperref}
\usepackage{booktabs}
\usepackage{color}
\usepackage[dvipsnames]{xcolor}
\usepackage{import}
\usepackage{tikz-cd}

\AtBeginDocument{%
   \def\MR#1{}
}

\usepackage{marginnote}
\long\def\@savemarbox#1#2{\global\setbox#1\vtop{\hsize\marginparwidth 
  \@parboxrestore\tiny\raggedright #2}}
\renewcommand*{\backref}[1]{}
\renewcommand*{\backrefalt}[4]{
  \ifcase #1
  [No citations.]
  \or [#2]
  \else [#2]
  \fi }

\numberwithin{equation}{section}
\theoremstyle{plain}
\newtheorem{theorem}[equation]{Theorem}
\newtheorem{corollary}[equation]{Corollary}
\newtheorem{lemma}[equation]{Lemma}

\newtheorem*{namedtheorem}{\theoremname}
\newcommand{\theoremname}{testing}
\newenvironment{named}[1]{\renewcommand{\theoremname}{#1}\begin{namedtheorem}}{\end{namedtheorem}}

\theoremstyle{definition}
\newtheorem{definition}[equation]{Definition}
\newtheorem{remark}[equation]{Remark}

\newcommand{\calT}{{\mathcal{T}}}

\newcommand{\cut}{\backslash\backslash}
\newcommand{\calH}{\mathcal{H}}
\newcommand{\calB}{\mathcal{B}}

\newcommand{\calX}{\mathcal{X}}

\makeatletter

\def\chaptermark#1{}%whatever

\def\chapter{%
  \if@openright\cleardoublepage\else\clearpage\fi
  \thispagestyle{plain}\global\@topnum\z@
  \@afterindenttrue \secdef\@chapter\@schapter}

\def\@chapter[#1]#2{\refstepcounter{chapter}%
  \ifnum\c@secnumdepth<\z@ \let\@secnumber\@empty
  \else \let\@secnumber\thechapter \fi
  \typeout{\chaptername\space\@secnumber}%
  \def\@toclevel{0}%
  \ifx\chaptername\appendixname \@tocwriteb\tocappendix{chapter}{#2}%
  \else \@tocwriteb\tocchapter{chapter}{#2}\fi
  \chaptermark{#1}%
  \addtocontents{lof}{\protect\addvspace{10\p@}}%
  \addtocontents{lot}{\protect\addvspace{10\p@}}%
  \@makechapterhead{#2}\@afterheading}
\def\@schapter#1{\typeout{#1}%
  \let\@secnumber\@empty
  \def\@toclevel{0}%
  \ifx\chaptername\appendixname \@tocwriteb\tocappendix{chapter}{#1}%
  \else \@tocwriteb\tocchapter{chapter}{#1}\fi
  \chaptermark{#1}%
  \addtocontents{lof}{\protect\addvspace{10\p@}}%
  \addtocontents{lot}{\protect\addvspace{10\p@}}%
  \@makeschapterhead{#1}\@afterheading}
\newcommand\chaptername{Chapter}

\def\@makechapterhead#1{\global\topskip 7.5pc\relax
  \begingroup
  \fontsize{\@xivpt}{18}\bfseries\centering
    \ifnum\c@secnumdepth>\m@ne
      \leavevmode \hskip-\leftskip
      \rlap{\vbox to\z@{\vss
          \centerline{\normalsize\mdseries
              \uppercase\@xp{\chaptername}\enspace\thechapter}
          \vskip 3pc}}\hskip\leftskip\fi
     #1\par \endgroup
  \skip@34\p@ \advance\skip@-\normalbaselineskip
  \vskip\skip@ }
\def\@makeschapterhead#1{\global\topskip 7.5pc\relax
  \begingroup
  \fontsize{\@xivpt}{18}\bfseries\centering
  #1\par \endgroup
  \skip@34\p@ \advance\skip@-\normalbaselineskip
  \vskip\skip@ }
\def\appendix{\par
  \c@chapter\z@ \c@section\z@
  \let\chaptername\appendixname
  \def\thechapter{\@Alph\c@chapter}}

\newcounter{chapter}

\newif\if@openright

\makeatother

\title{Some fast algorithms for curves in surfaces}
\author{Marc Lackenby}
\address{Mathematical Institute, University of Oxford, \newline Woodstock Road, Oxford OX2 6GG, United Kingdom}
\thanks{The author was partially supported by EPSRC grant EP/Y004256/1. For the purpose of open access, the author has applied a CC BY public copyright licence to any author accepted manuscript arising from this submission.}

\begin{document}

\begin{abstract} 
We present some algorithms that provide useful topological information about curves in surfaces. One of the main algorithms
computes the geometric intersection number of two properly embedded 1-manifolds $C_1$ and $C_2$ in a compact orientable surface $S$. The surface $S$ is presented via a triangulation or a handle structure,
and the 1-manifolds are given in normal form via their normal coordinates. The running time is bounded above by a polynomial function
of the number of triangles in the triangulation (or the number of handles in the handle structure), and the logarithm of the weight of $C_1$ and $C_2$.
This algorithm represents an improvement over previous work, since its running time depends polynomially on the size of the triangulation of $S$ and it can 
deal with closed surfaces, unlike many earlier algorithms. Another algorithm, with similar bounds on its running time, can determine whether $C_1$ and $C_2$ are isotopic.
We also present a closely related algorithm that can be used to place a standard 1-manifold into normal form.
\end{abstract}
\maketitle

\section{Introduction}
\label{Sec:Intro}
Surfaces occupy a central position in low-dimensional topology and geometry. They are mostly studied by means of properly embedded 1-manifolds (that is, simple closed curves and arcs, that are pairwise disjoint, and with the endpoints of the arcs in the boundary of the surface). A foundational result is that two properly embedded essential 1-manifolds $C_1$ and $C_2$ in a compact surface $S$ can be isotoped into an essentially unique position where $|C_1 \cap C_2|$ is minimal. The \emph{geometric intersection number} $i(C_1, C_2)$ is then $|C_1 \cap C_2|$. This minimal position for $C_1$ and $C_2$ is easy to compute in principle, because when $C_1$ and $C_2$ are not in minimal position, then there is a bigon or half-bigon (as shown in Figure \ref{Fig:BigonHalfBigon}) in their complement, and this can be used to reduce the number of points of intersection. However, when $C_1$ and $C_2$ intersect many times, the computation of their geometric intersection number is not at all straightforward. Nevertheless, the following result shows that it can be computed extremely efficiently.

\begin{theorem}[{\sc Geometric intersection number}]
\label{Thm:GeometricIntersectionNumber}
Let $\mathcal{T}$ be a triangulation for a compact orientable surface $S$. Let $C_1$ and $C_2$ be normal 1-manifolds properly embedded in $S$. Then there is an algorithm that provides the geometric intersection number of $C_1$ and $C_2$. The running time of this algorithm is bounded above by a polynomial function of $|\mathcal{T}|$, $\log w(C_1)$ and $\log w(C_2)$.
\end{theorem}

In this result, and throughout this paper, we require the surface $S$ to be compact. It may be closed or have non-empty boundary.  In this result, we also require $S$ to be orientable, but in some later theorems, we can drop this hypothesis. 
Here, $|\mathcal{T}|$ is the number of triangles of $\mathcal{T}$. Recall that a 1-manifold $C$ in $S$ is \emph{normal} if it intersects each triangle in properly embedded arcs, each of which has endpoints in the interior of distinct edges of the triangle. The number $w(C)$ is the \emph{weight} of a 1-manifold $C$, which is its number of points of intersection with the 1-skeleton of $\calT$. Note that in the case where $\mathcal{T}$ is a 1-vertex triangulation for a closed surface $S$, then $|\mathcal{T}|$ is equal to $2 - 2\chi(S)$ and hence the upper bound on the running time depends polynomially on $\chi(S)$. In the case where $S$ has non-empty boundary and $\mathcal{T}$ has one vertex on each boundary component and no vertices in the interior of $S$, then the upper bound on the running time again depends polynomially on $\chi(S)$.

\begin{figure}[h]
\centering
\includegraphics[width=0.35\textwidth]{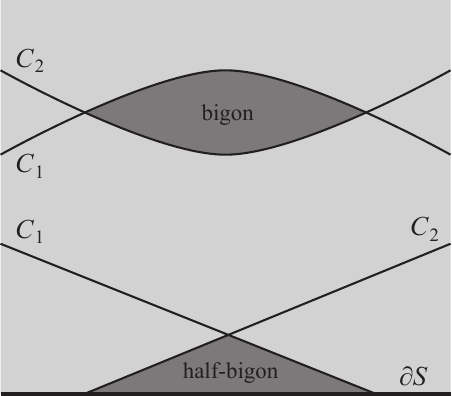}
\caption{A bigon and a half-bigon for two properly embedded 1-manifolds $C_1$ and $C_2$}
\label{Fig:BigonHalfBigon}
\end{figure}

The computation of geometric intersection number has been studied in detail by several authors, including Bell and Webb \cite{BellWebb}, Dynnikov \cite{Dynnikov},
Schaefer, Sedgwick and \u{S}tefankovi\u{c} \cite{SSS2}, Despr\'e and Lazarus \cite{DespreLazarus}, Dubois \cite{dubois} and Chang and de Mesmay \cite{ChangdeMesmay}. In Section \ref{Subsec:PreviousWork}, we will compare those works with ours, highlighting what is new in Theorem \ref{Thm:GeometricIntersectionNumber}.

%However, the first three of these algorithms require that $S$ has non-empty boundary. This is because they use ideal triangulations and they rely on the fact that a properly embedded essential curve in such a surface has a \emph{unique} normal representative with respect to an ideal triangulation. There is no such uniqueness result in the case of closed surfaces, and so their methods break down in this important case. Another significant distinction between Theorem \ref{Thm:GeometricIntersectionNumber} and the work in \cite{BellWebb} and \cite{SSS2} is that our algorithm has running time that depends only polynomially on $|\calT|$. On the other hand, the algorithms of Despr\'e and Lazarus and of Chang and de Mesmay have polynomial dependence of the genus of the surface, but also depend polynomially on the initial number of intersection points of the given 1-manifolds $C_1$ and $C_2$, and as a consequence, they do not depend polynomially on $\log w(C_1)$ and $\log w(C_2)$.

In a similar vein, we have the following result.

\begin{theorem}[{\sc Isotopy of 1-manifolds}]
\label{Thm:IsotopyProblem}
Let $\mathcal{T}$ be a triangulation for a compact orientable surface $S$. Let $C_1$ and $C_2$ be normal essential 1-manifolds properly embedded in $S$. Then there is an algorithm that determines whether $C_1$ and $C_2$ are isotopic. The running time of this algorithm is bounded above by a polynomial function of $|\mathcal{T}|$, $\log w(C_1)$ and $\log w(C_2)$.
\end{theorem}

A similar result was proved by Schleimer \cite[Theorem A.7]{Schleimer:Word}, who provided a new solution to the conjugacy problem for $\pi_1(S)$. Two simple closed curves represent conjugate elements of $\pi_1(S)$ if and only if they are isotopic. Schleimer's algorithm takes, as its input, `compressed' words in $\pi_1(S)$ and hence runs in polynomial time as a function of $\log w(C_1)$ and $\log w(C_2)$ with respect to a suitable triangulation $\calT$. But how the running time in Schleimer's algorithm depends on the triangulation $\mathcal{T}$ and, more generally on $\chi(S)$, is not explicitly addressed in \cite{Schleimer:Word}. In \cite[Theorem 2(a)]{SSS1}, another algorithm was given to determine whether two simple closed curves are isotopic, but it requires $S$ to have non-empty boundary and the dependence of its running time on $\chi(S)$ is unclear.

Our next main result gives an algorithm that places a 1-manifold into normal form. For simplicity, we initially consider the case of surfaces with a given triangulation. Later, we will alternatively endow our surfaces with a handle structure.

A 1-manifold $C$ properly embedded in a surface $S$ with a triangulation $\calT$ is \emph{standard} if it misses the vertices of $\calT$, is transverse to the edges of $\calT$ and intersects each triangle in a collection of properly embedded arcs. A \emph{simplifying disc} for $C$ is a disc $D$ lying in a triangle $\Delta$ of the triangulation such that
\begin{enumerate}
\item $\partial D$ is the union of two arcs $\alpha$ and $\beta$ that intersect at their endpoints;
\item $\alpha = D \cap C$;
\item $\beta = D \cap \partial \Delta$ is a subset of an edge of $\Delta$.
\end{enumerate}
See Figure \ref{Fig:SimplifyingCurve}.
The \emph{simplification number} $\mathrm{simp}(C)$ is the number of simplifying discs for $C$.

\begin{figure}[h]
\centering
\includegraphics[width=0.5\textwidth]{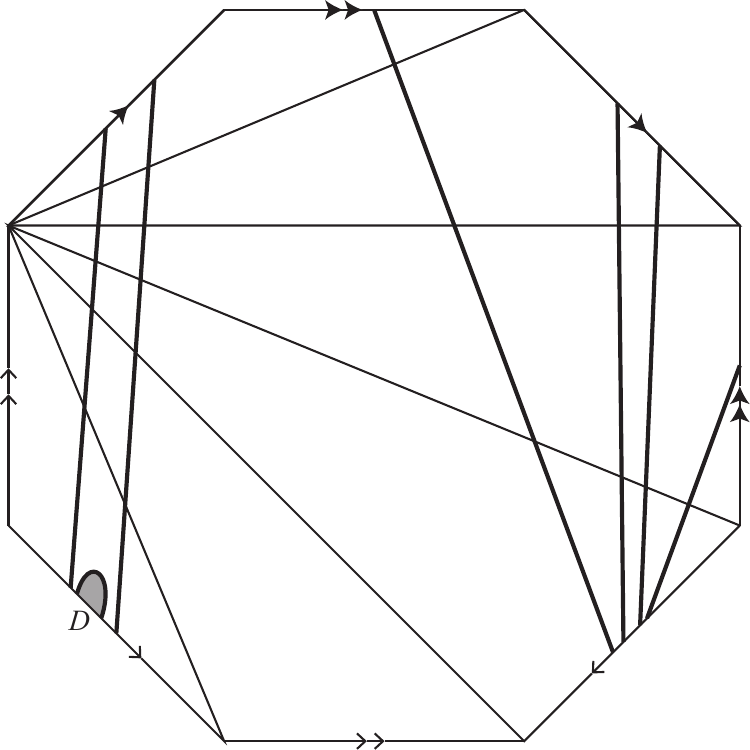}
\caption{A standard curve in a triangulated surface, with a simplifying disc $D$}
\label{Fig:SimplifyingCurve}
\end{figure}

\begin{theorem}[{\sc Normalise standard 1-manifold in a triangulation}]
\label{Thm:NormaliseStandardCurves}
There is an algorithm that takes, as its input, the following data:
\begin{enumerate}
\item a triangulation $\mathcal{T}$ for a compact surface $S$;
\item a standard 1-manifold $C$ properly embedded in $S$;
\end{enumerate}
and provides, as its output, a normal 1-manifold $\overline{C}$ that is obtained from $C$ by an ambient isotopy and possibly removing some components that are parallel to arcs in $\partial S$ or that bound discs in $S$. The running time for the algorithm is bounded above by a polynomial function of the number of triangles in $\mathcal{T}$, the simplification number of $C$ and the logarithm of the weight of $C$.
\end{theorem}

In order for the statement of the theorem to be complete, we need to explain how a standard 1-manifold in a surface is entered into the algorithm. 
To do this, we first explain how to describe a system of disjoint properly embedded arcs $C$ in a triangle $\Delta$, where these arcs are disjoint from the vertices of $\Delta$. Suppose that the edges of $\Delta$ are oriented in some way. The arcs intersect the edges some numbers of times, $n_1$, $n_2$ and $n_3$. Since each edge is oriented, we may identify these points of intersection with the $i$th edge with the integers $[1,n_i] \cap \mathbb{Z}$. We say that two arcs $\alpha_0$ and $\alpha_1$ of $C$ are \emph{parallel} if there is an embedding of $[0,1] \times [0,1]$ in $\Delta$ such that 
\begin{enumerate}
\item $[0,1] \times \{ 0 \} = \alpha_0$, 
\item $[0,1] \times \{ 1 \} = \alpha_1$, 
\item $([0,1] \times [0,1]) \cap \partial \Delta = \{ 0,1 \} \times [0,1]$,
\item this square is disjoint from the vertices of $\Delta$,
\item the intersection between $[0,1] \times [0,1]$ and each component of $C$ is either empty or an arc of the form $[0,1] \times \{ t \}$. 
\end{enumerate}
This forms an equivalence relation on the components of $C$ and we refer to the equivalence classes as \emph{parallelism classes}. The endpoints of the arcs in a parallelism class consist of two intervals $[a,b] \cap \mathbb{Z}$ and $[c,d] \cap \mathbb{Z}$ in the intervals $[1,n_i] \cap \mathbb{Z}$ and $[1,n_j] \cap \mathbb{Z}$ corresponding to the relevant edges, and the arcs themselves determine an order-preserving or order-reversing bijection $[a,b] \cap \mathbb{Z} \rightarrow [c,d] \cap \mathbb{Z}$, which is termed a \emph{pairing}. The collection of all these pairings is termed a \emph{pairing system} for $\Delta$. See Figure \ref{Fig:PairingSystem}. It is clear that there is a simple procedure for determining whether a collection of pairings determines a pairing system. In fact, a collection of pairings gives a pairing system if and only if the following all hold:
\begin{enumerate}
\item the ranges and domains of all the pairings form a partition of $[1,n_1] \cap \mathbb{Z}$, $[1,n_2] \cap \mathbb{Z}$ and $[1,n_3] \cap \mathbb{Z}$; 
\item a pairing is order-reversing if and only if the orientations of the edges containing the range and domain are consistent around $\partial \Delta$; 
\item the ranges and domains of no two pairings interleave around $\partial \Delta$.
\end{enumerate}
Thus, a standard 1-manifold $C$ in a triangulated surface is given by the following data:
\begin{enumerate}
\item an orientation on each edge;
\item a non-negative integer $n(e)$ for each edge $e$; this is the number of points of intersection between $C$ and $e$;
\item a pairing system for each face, where the number of endpoints of the arcs on each edge $e$ is $n(e)$.
\end{enumerate}
The above definition of a pairing system in a triangle readily generalises to the case of polygons, in an obvious way.

\begin{figure}[h]
\centering
\includegraphics[width=0.5\textwidth]{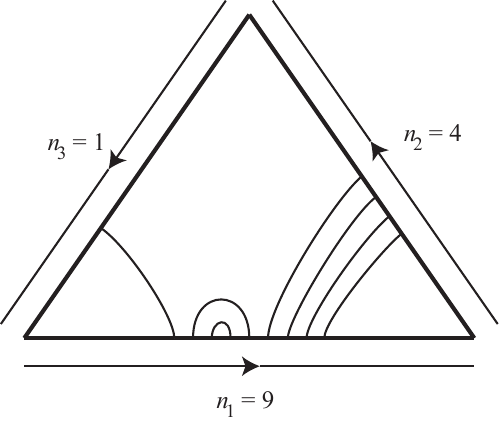}
\caption{Shown is a system of arcs in triangle. This system is specified by picking orientations on the edges, as shown, and setting $n_1 = 9$, $n_2 = 4$ and $n_3 = 1$. The arcs are given by three pairings. For example, one set of arcs is given by the pairing $[2, 3] \cap \mathbb{Z} \rightarrow [4,5] \cap \mathbb{Z}$, sending $x$ to $7-x$, and where the domain and range are both viewed as subsets of $[1,n_1] \cap \mathbb{Z}$.}
\label{Fig:PairingSystem}
\end{figure}

\subsection{Comparison with previous work}
\label{Subsec:PreviousWork}
As mentioned earlier, there has been a lot of previous work on the computation of geometric intersection number of curves in surfaces, which we now briefly summarise. 

In  \cite{SSS1}, Schaefer, Sedgwick and \u{S}tefankovi\u{c} used the theory of straight line programs to provide very efficient solutions to many computational problems about 1-manifolds properly embedded in surfaces. They built on this theory in \cite{SSS2} to provide an algorithm to compute the geometric intersection number of two properly embedded 1-manifolds. One significant difference between their work and ours is that they restricted to the case where $S$ has non-empty boundary. It turns out when $\partial S$ is non-empty, many algorithmic questions become somewhat easier. One way of understanding this is to note that the fundamental group of $S$ is then free, and so closed curves then have a \emph{unique} cyclically reduced representative in $\pi_1(S)$. Our algorithms, on the other hand, permit $S$ to be closed. Another significant distinction between Theorem \ref{Thm:GeometricIntersectionNumber} and the work in \cite{SSS2} is that our algorithm has running time that depends only polynomially on $|\calT|$. 

An alternative approach to computing geometric intersection numbers was given by Bell and Webb in \cite{BellWebb}. This relied on prior work of Bell \cite{Bell}, which provided a rapid method for modifying triangulations of surfaces. Given two 1-manifolds $C_1$ and $C_2$ properly embedded in $S$ expressed as normal curves with respect to an ideal triangulation $\calT$, Bell and Webb were able to change $\calT$ to another ideal triangulation $\calT'$ that intersects $C_1$ in a very simple form. They could also keep track of how $C_2$ lies with respect to $\calT'$, and then ensure that $C_2$ and $C_1$ have no complementary bigon. The geometric intersection number of $C_1$ and $C_2$ could then be directly read off. Like in \cite{SSS2}, Bell and Webb require $S$ to have non-empty boundary, since they make substantial use of the fact that an essential simple closed curve has a unique normal representative with respect to an ideal triangulation. Also as in \cite{SSS2}, they make no analysis of how the running time depends on $S$. Nevertheless, the approach taken by Bell and Webb is possibly closest to the one we use, since we also use modifications of our given triangulation.

A third approach to geometric intersection numbers was by Dynnikov \cite{Dynnikov}, who introduced a novel method of representing embedded curves on surfaces. He also also used a novel method of expressing mapping classes of surfaces via matrices, so that composition of mapping classes can be computed via matrix multiplication. However, again, Dynnikov required the surface $S$ to have non-empty boundary.

Despr\'e and Lazarus \cite{DespreLazarus} analysed a related but different problem. They considered 1-manifolds $C$ in the surface $S$ with $C \cap \partial S = \partial C$, but which might not be embedded. Such a curve has a well-defined self-intersection number, which is the minimal number of self-intersections among representatives with only double points, up to homotopy that keeps $\partial C$ in $\partial S$. They give an efficient method for computing this self-intersection number. 
Despr\'e and Lazarus had to focus on a connected 1-manifold. On the other hand, Chang and de Mesmay \cite{ChangdeMesmay} and Dubois \cite{dubois} considered the same problem, but for 1-manifolds with multiple components. By applying this to a 1-manifold that is the union of two embedded 1-manifolds $C_1$ and $C_2$, this gives another method for computing their geometric intersection number. The algorithms of Despr\'e and Lazarus, Dubois and Chang and de Mesmay have polynomial dependence of the genus of the surface and permit the surface to be closed, but also depend polynomially on the initial number of self-intersection points of the given 1-manifold $C$. Hence, when applied to the union of two 1-manifolds $C_1$ and $C_2$, the running time does not depend polynomially on $\log w(C_1)$ and $\log w(C_2)$.

The methods that we present in this paper rely (like \cite{SSS2, BellWebb, Dynnikov}) very heavily on the assumption that the 1-manifolds are embedded. Nevertheless, there is some hope that our methods may be adapted and possibly combined with those of \cite{DespreLazarus, dubois, ChangdeMesmay} to efficiently compute the geometric intersection number of two 1-manifolds $C_1$ and $C_2$, which may be non-embedded but which have a relatively small number of self-intersections.

\subsection{Structure of the paper} In Section \ref{Sec:HSNormal}, we introduce handle structures for surfaces and normal and standard 1-manifolds within them. We also state versions of Theorems \ref{Thm:GeometricIntersectionNumber} and \ref{Thm:NormaliseStandardCurves} for handle structures instead of triangulations. In Section \ref{Sec:AHT}, we recall the algorithm of Agol, Hass and Thurston \cite{AHT} and explain some of its uses in the study of normal and standard 1-manifolds. In Section \ref{Sec:Cut}, we examine the handle structure obtained by cutting along a standard 1-manifold, and explain how it can be computed efficiently. In Section \ref{Sec:Normalise}, we prove Theorem \ref{Thm:NormaliseStandardCurves} and its analogue for handle structures. This section is the heart of the paper, as its main ideas recur at several points later. In Section \ref{Sec:MinimalPositionCurves}, we show how to place two 1-manifolds $C$ and $P$ into minimal position, via an isotopy of $C$, in the important case where $P$ is simplicial in the triangulation. In Section \ref{Sec:GeometricIntersectionNumber},
this is used, along with a theorem of the author and Yazdi \cite{LackenbyYazdi}, to complete the proof of Theorems \ref{Thm:GeometricIntersectionNumber} and \ref{Thm:IsotopyProblem}. In Section \ref{Sec:Pattern}, we consider a more general situation, where the surface $S$ contains a `pattern' which is a union of disjoint properly embedded 1-manifolds and graphs (with restrictions on the vertex degrees). We show how to arrange for a properly embedded 1-manifold to intersect a pattern `minimally'. The reason for studying patterns is that, in the analysis of hierarchies for 3-manifolds, the boundary of each 3-manifold in the hierarchy naturally inherits a pattern \cite{HakenHomeomorphism, Matveev}.

\section{Handle structures and normal 1-manifolds}
\label{Sec:HSNormal}

\subsection{Handle structures}
In Theorems \ref{Thm:GeometricIntersectionNumber} and \ref{Thm:NormaliseStandardCurves}, the surface $S$ was given by a triangulation. This is because triangulations are a very typical way of presenting a surface. However, it is frequently useful to deal with handle structures instead. (See, for example, Definition \ref{Def:InheritedHS} where a useful construction involving handle structures is presented, and which has no direct analogue for triangulations.) As we will see, there are versions of Theorems \ref{Thm:GeometricIntersectionNumber} and \ref{Thm:NormaliseStandardCurves} for handle structures. 
In this section, we introduce some of the terminology for handle structures. 

\begin{definition}
A \emph{handle structure} for a surface $S$ is a decomposition of $S$ into three sets $\calH^0$, $\calH^1$, $\calH^2$, subject to the following conditions:
\begin{enumerate}
\item $\calH^0$ is a union of disjoint closed discs;
\item each component of $\calH^1$ is a disc of the form $[0,1] \times [0,1]$ with $([0,1] \times [0,1]) \cap \calH^0 = \{ 0 , 1 \} \times [0,1]$;
\item each component of $\calH^2$ is a disc $D$ with $\partial D = D \cap (\calH^0 \cup \calH^1)$.
\end{enumerate}
Each component of $\calH^i$ is called an \emph{$i$-handle}. See, for example, Figure \ref{Fig:HSPants}. For each 1-handle $[0,1] \times [0,1]$, the arc $[0,1] \times \{ 1/2 \}$ is its \emph{core} and $\{ 1/2 \} \times [0,1]$ is its \emph{co-core}.
\end{definition}
%For a handle structure $\calH$ and $i \in \{ 0,1,2 \}$, $\calH^i$ will denote the union of the $i$-handles.

\begin{figure}[h]
\centering
\includegraphics[width=0.7\textwidth]{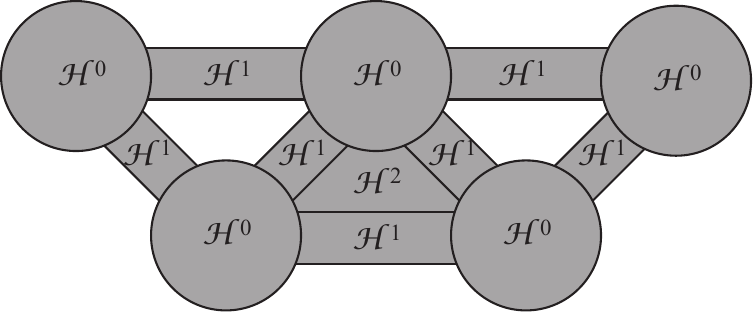}
\caption{A handle structure on a 3-holed sphere}
\label{Fig:HSPants}
\end{figure}

\begin{definition}
Let $\calH$ be a handle structure for a surface $S$. Its \emph{size} $||\calH||$ is the number of handles.
\end{definition}

\begin{remark}
It is quite common to work instead with the number of 0-handles of $\calH$, denoted $|\calH|$. This is only a good measure of the complexity of $\calH$ when the 0-handles are constrained to fall into one of finitely many pre-arranged types. We do not wish to make that hypothesis in this paper, and we therefore work instead with $||\calH||$.
\end{remark}

\begin{definition}
Let $\calH$ be a handle structure for a surface $S$. Its \emph{underlying cell structure} is the cell structure on $S$ where each handle is a 2-cell. Each 1-cell arises either as a component of intersection between two handles or between a handle and $\partial S$. Each 0-cell arises either as a component of intersection between three handles or between two handles and $\partial S$.
\end{definition}

For a cell structure $\calX$ and $i \in \{ 0,1,2 \}$, $\calX^i$ denotes the set of $i$-cells.

\begin{lemma} 
\label{Lem:Number1Cells}
The number of 1-cells in the underlying cell structure is $6|\calH^1| \leq 6||\calH||$.
\end{lemma}

\begin{proof}
Each 1-cell lies in the boundary of a 0-handle or a 1-handle. The number of 1-cells in the boundary of a 1-handle but not the boundary of a 0-handle is $2|\calH^1|$. The number of 1-cells in the boundary of a 0-handle is $4|\calH^1|$.
\end{proof}

A cell structure on a surface $S$ is specified by giving a list of cells, together with their attaching maps onto the lower-dimensional cells. These attaching maps are specified as follows. Each 1-cell is oriented and its initial and terminal 0-cells are specified. The boundary of each 2-cell is specified as a sequence of 1-cells, together with information about whether each 1-cell is traversed consistently or inconsistently with the orientation. 

A handle structure $\calH$ on a surface $S$ is specified by giving its underlying cell structure $\calX$, together with the index of each handle. The amount of data required is then bounded above by a polynomial function of $||\calH||$ for the following reason. The attaching maps of the 1-cells require $2|\calX^1| + ||\calH||$ pieces of information, each of which is either the indexing number of a 0-cell or a spacing digit specifying a break between one 2-cell and the next. We may give each indexing number using at most $\log (||\calH|| + 1)$ digits. The attaching maps of the 2-cells require at most $2|\calX^1|$ pieces of information, each of which is the number of a 1-cell as well as its orientation. The number of 1-cells of $\calX$ is at most $6||\calH||$, by Lemma \ref{Lem:Number1Cells}. So the amount of data required to specify $\calH$ is at most $O(||\calH|| \log (||\calH||+1))$ bits and at least $||\calH||$ bits.

\subsection{Standard and normal 1-manifolds}

\begin{definition} 
Let $S$ be a surface with a handle structure $\calH$. Then a 1-manifold $C$ properly embedded in $S$ is \emph{standard} if
\begin{enumerate}
\item it intersects each 0-handle in a collection of properly embedded arcs;
\item it intersects each 1-handle in a collection of properly embedded arcs that respect the handle's product structure and run parallel to the core of the handle;
\item it is disjoint from the 2-handles.
\end{enumerate}
\end{definition}

\begin{definition}
The \emph{weight} $w(C)$ of a standard 1-manifold $C$ is the number of components of $C \cap \calH^1$.
\end{definition}

\begin{remark} 
For simple closed curves $C$, $w(C)$ is a reasonable measure of their complexity. However, when $C$ has arc components, some of these may have weight zero. Therefore, $w(C) + |\partial C|$ is a better measure of the complexity of $C$.
\end{remark}

\begin{definition} 
Let $\calH$ be a handle structure for a compact surface $S$. Let $\calX$ be the associated cell structure of $S$, with a 2-cell for each handle. Then a \emph{pairing system} for a standard 1-manifold $C$ is the following:
\begin{enumerate}
\item it misses $H_0 \cap \calH^1 \cap \calH^2$ and $H_0 \cap \calH^1 \cap \partial S$;
\item a choice of orientation on each 1-cell in $\calX$;
\item a non-negative integer for each 1-cell that gives the number of points of intersection
between $C$ and the 1-cell;
\item a pairing system on each 2-cell of $\calH^0$ and $\calH^1$, that respects the product structure on each component of $\calH^1$, and that agrees on the boundary of the 2-cell with the data given in (2).
\end{enumerate}
\end{definition}

\begin{definition}
Let $\calH$ be a handle structure for a compact surface $S$, and let $H_0$ be a 0-handle of $\calH$. An arc properly embedded in $H_0$ is \emph{normal} if it satisfies the following conditions:
\begin{enumerate}
\item it does not have endpoints in the same component of $H_0 \cap \calH^1$;
\item it does not have endpoints in the same component of $H_0 \cap \partial S$;
\item it does not have one endpoint in $H_0 \cap \calH^1$ and the other endpoint in an adjacent component of $H_0 \cap \partial S$.
\end{enumerate}
\end{definition}

\begin{definition}
Let $\calH$ be a handle structure for a compact surface $S$. A 1-manifold $C$ properly embedded in $S$ is \emph{normal} if it is standard and, in addition, each component of $C \cap \calH^0$ is normal.
\end{definition}

\begin{definition}
\label{Def:SimpInHS}
Let $S$ be a surface with a handle structure $\calH$. Let $C$ be a standard 1-manifold properly embedded in $S$. Then a \emph{simplifying disc} for $C$ is a disc $D$ lying in a 0-handle $H_0$ such that $\partial D$ is the union of two arcs $\alpha$ and $\beta$ that intersect at their endpoints, satisfying the following:
\begin{enumerate}
\item $\alpha = D \cap C$; 
\item $\beta = D \cap \partial H_0$;
\item $\beta$ intersects $\calH^1$ in at most one arc;
\item $\beta$ intersects $\partial S$ in at most one arc.
\end{enumerate}
If $D$ intersects $\partial S$, it is a \emph{boundary-simplifying disc}. Otherwise it is an \emph{interior-simplifying disc}.
\end{definition}

\begin{figure}[h]
\centering
\includegraphics[width=0.9\textwidth]{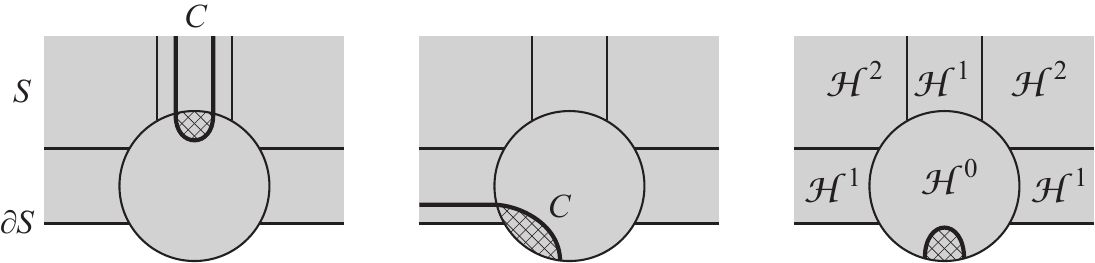}
\caption{The three types of a simplifying disc for a 1-manifold in a surface. On the left is an interior-simplifying disc. In the middle and on the right are boundary-simplifying discs}
\label{Fig:IncreaseSimplifyingIndex}
\end{figure}

\begin{lemma}
\label{Lem:NormalNoSimplifying}
Let $\calH$ be a handle structure for a compact surface $S$, and let $C$ be a standard 1-manifold properly embedded in $S$. Then $C$ is normal if and only if it has no simplifying disc.
\end{lemma}

\begin{proof} If $C$ has a simplifying disc $D$, then $D \cap C$ is a non-normal arc. Conversely, if $C$ has a non-normal arc $\alpha$ in some 0-handle $H_0$, then some component $D$ of $H_0 \cut \alpha$ is a disc $D$ that intersects $\calH^1$ in at most one arc and $\partial S$ in at most one arc. This disc may contain arcs of intersection with $C$, but then an outermost such arc separates off a simplifying disc for $C$.
\end{proof}

%satisfying one of the following conditions:
%\begin{enumerate}
%\item $\partial D$ is the concatenation of an arc $\alpha = D \cap F$ and an arc $\beta = D \cap \partial H_0$, where $\beta$ lies in $\calH^0 \cap \calH^1$;
%\item $\partial D$ is the concatenation of three arcs $\alpha$, $\beta$ and $\gamma$, where $\alpha = D \cap F$, $\beta \cup \gamma = D \cap \partial H_0$, $\beta$ is an arc in $\calH^0 \cap \calH^1$ and $\gamma = D \cap \partial F$.
%\end{enumerate}

\begin{definition}
\label{Def:SimpNumber}
The \emph{simplification number} $\mathrm{simp}(C)$ of $C$ is the sum of the number of boundary-simplifying discs and twice the number of interior-simplifying discs.
\end{definition}

In the proof of Lemma \ref{Lem:SimpDoesNotGoUp}, we explain why boundary-simplifying discs and interior-simplifying discs are given different weightings in the above definition.

The following result is an analogue of Theorem \ref{Thm:NormaliseStandardCurves} for handle structures, and is one of the main theorems of this paper.

\begin{theorem}[{\sc Normalise a standard 1-manifold in a handle structure}]
\label{Thm:Normalise1Manifold}
There is an algorithm that takes as its input the following:
\begin{enumerate}
\item a handle structure $\calH$ for a compact surface $S$;
\item a standard 1-manifold $C$ properly embedded in $S$;
\end{enumerate}
and outputs a normal 1-manifold $\overline{C}$ that is obtained from $C$ by an isotopy plus possibly discarding some curves that bound discs and boundary-parallel arcs. The running time for the algorithm is bounded above by a polynomial function of $||\calH||$, $\mathrm{simp}(C)$ and $\log (w(C) + |\partial C|)$.
\end{theorem}

\subsection{The size of a pairing system}
\label{SubSec:SizePairing}

Pairing systems are a very natural way of specifying standard 1-manifolds in surfaces.
In this section, we investigate the amount of data required to specify a pairing system.
%We also present an alternative method of specifying 1-manifolds that has more in common with normal surface theory.

We will be given a handle structure $\calH$ for a surface $S$. We will also be given a standard 1-manifold $C$, as a pairing system. Recall from above that this means that each handle is viewed as a polygon, and then a pairing system is given in each polygon. We therefore introduce the following definition.

\begin{definition}
Let $P$ be a polygon, and let $C$ be a union of disjoint properly embedded arcs in $P$, disjoint from the vertices of $P$. Then a \emph{simplifying disc} for $C$ is a disc $D$ embedded in $P$ satisfying the following:
\begin{enumerate}
\item $\partial D$ is the union of two arcs $\alpha$ and $\beta$ that intersect at their endpoints;
\item $\alpha = D \cap C$;
\item $\beta = D \cap \partial P$ is a subset of an edge of $P$.
\end{enumerate}
The simplification number $\mathrm{simp}(C,P)$ is the number of simplifying discs.
\end{definition}

When we view a 0-handle $H_0$ of $\calH$ as a polygon, then any simplifying disc for $C \cap H_0$ in the above sense is a simplifying disc for $C$ in the sense of Definition \ref{Def:SimpNumber}. Hence, $\mathrm{simp}(C) \geq \mathrm{simp}(C, \calH^0)$, where the latter is the sum of $\mathrm{simp}(C, H_0)$ over each 0-handle $H_0$.

%In general, a pairing system is just a collection of pairings $g_i \colon [a_i, b_i] \cap \mathbb{Z} \rightarrow [c_i, d_i] \cap \mathbb{Z}$ and so this is just specified by the integers $a_i, b_i, c_i, d_i$ given in binary as well as the information about whether it is order-preserving or reversing. However, in our setting, the domain and ranges are associated with 1-cells of the cells structure, and these 1-cells also need to be specified, as well as their orientations.

The following lemma bounds the number of parallelism classes of arcs in a polygon in terms of their simplification number.

\begin{lemma}
\label{Lem:ArcTypesInPolygon}
Let $P$ be a polygon with $k$ sides. Let $C$ be a non-empty union of disjoint properly embedded arcs. Then the number of parallelism classes among the components of $C$ is at most
$2(k + \mathrm{simp}(C,P)) - 3$.
\end{lemma}

\begin{proof}
We start by subdividing the edges of the polygon $P$, by placing an extra vertex in each simplifying disc for $C$. The resulting polygon $P'$
has $v = k + \mathrm{simp}(C,P)$ vertices. Let $\overline{C}$ be the arcs obtained by taking one representative of each arc type of $C$.
Note that $\overline{C}$ has no simplifying discs in $P'$.
Then $P'$ is cut into discs by $\overline{C}$. For each such disc, define its \emph{complexity} to be the number of arcs of $\overline{C}$ in its boundary, 
plus twice the number of vertices of $P'$ that it contains, minus $2$. Then the sum of the complexities of these discs is equal to $2 |\overline{C}| + 2 v - 2 |P' \cut \overline{C}| = 2v - 2$. 
A disc with negative complexity would have a single arc of $\overline{C}$ in its boundary and no vertices of $P'$ and hence would be a simplifying disc for $\overline{C}$, which does not exist. 
A disc with zero complexity must have two arcs of $\overline{C}$ in its boundary and no vertices,
and therefore these two arcs are parallel. However, by construction, no two arcs of $\overline{C}$ are parallel. Hence, the number of components of 
$P' \cut \overline{C}$ is at most $2v-2$. So, the number of components of $\overline{C}$ is at most $2v-3$.
\end{proof}

The following lemmas are immediate consequences.

\begin{lemma} 
\label{Lem:NumberOfPairingsCurve}
Let $\calH$ be a handle structure for a compact surface $S$. 
Let $C$ be a standard 1-manifold properly embedded in $S$. Then the number of pairings in its associated pairing system
is bounded above by $9|\calH^1| + 2\mathrm{simp}(C) \leq 9||\calH|| + 2\mathrm{simp}(C)$.
\end{lemma}

\begin{proof} 
Note that $C$ intersects only the 0-handles and 1-handles of $\calH$. In each 1-handle, there is at most one parallelism class of arcs in $C$. So we focus on the 0-handles.
Each 0-handle $H_0$ can be viewed as a polygon with $2|\calH^1 \cap H_0|$ sides. Each 1-handle has two components of intersection with the 0-handles. Hence, the sum, over all 0-handles, of the number of sides is $4|\calH^1|$. So, by Lemma \ref{Lem:ArcTypesInPolygon}, the number of arc types of $C$ in the 0-handles is at most $8|\calH^1| + 2 \mathrm{simp}(C,\calH^0)$. As observed above, $\mathrm{simp}(C,\calH^0) \leq \mathrm{simp}(C)$. The lemma follows immediately.
%Hence, the number of handles of the inherited handle structure on $S \cut C$ that are not parallelity handles is at most $9||\calH|| + 2 \mathrm{simp}(C)$.
\end{proof}

\begin{definition}
Let $\calH$ be a handle structure for a compact surface $S$. 
Let $C_0$ and $C_1$ be disjoint connected standard 1-manifolds properly embedded in $S$. Then they are \emph{normally parallel} if there is an 
embedding of $C \times [0,1]$ in $S$ for some 1-manifold $C$, such that $C \times \{ 0 \} = C_0$, $C \times \{ 1 \} = C_1$, and for each $t \in [0,1]$,
$C \times \{ t \}$ is standard and properly embedded.
\end{definition}

\begin{lemma}
\label{Lem:NumberOfCurveComponents}
Let $\calH$ be a handle structure of a compact surface $S$. Let $C$ be a standard 1-manifold properly embedded in $S$. Suppose that no two components of $C$ are normally parallel. Then $|C|$ is bounded above by $10 ||\calH|| + 2 \mathrm{simp}(C)$.
\end{lemma}

\begin{proof}
It is a general fact that $|S \cut C| \geq |C| - b_1(S;\mathbb{Z}_2)$. The reason is that if we cut $S$ along $C$ one component at a time, then we either increase the number of complementary regions or we find an element of $H^1(S;\mathbb{Z}_2)$ that is linearly independent from previously found elements. Furthermore, $b_1(S;\mathbb{Z}_2)$ is bounded above by $|\calH^1| \leq ||\calH||$. As argued in the proof of  Lemma \ref{Lem:NumberOfPairingsCurve}, the number of arc types of $C \cap \calH^0$ is bounded above $8|\calH^1| + 2\mathrm{simp}(C)$. Hence, the number of 0-handles of $S \cut C$ that do not lie between parallel normal arcs of $C$ is at most $|\calH^0| + 8 |\calH^1| + 2\mathrm{simp}(C)$. Consider any component of $S \cut C$ that does not contain at least one such 0-handle. It cannot be an annulus or disc, since it would then co-bound normally parallel components of $C$. Hence, it must be a M\"obius band, which therefore contributes to $b_1(S;\mathbb{Z}_2)$. Therefore, $|S \cut C|$ is at most $8||\calH|| + 2\mathrm{simp}(C) + b_1(S;\mathbb{Z}_2)$. We deduce that $|C|$ is bounded above by $10||\calH|| + 2\mathrm{simp}(C)$.
\end{proof}

%\begin{lemma}
%\label{Lem:NumberOfCurveComponentsTriangulation}
%Let $\calT$ be a triangulation of a compact surface $S$. Let $C$ be a normal 1-manifold properly embedded in $S$. Suppose that no two components of $C$ are normally parallel. Then $|C|$ is bounded above by
%\end{lemma}

%\begin{proof}
%It is a general fact that $|S \cut C| \geq |C| - b_1(S)$. The reason is that if we cut $S$ along $C$ one component at a time, then we either increase the number of complementary regions or we find an element of $H^1(S)$ that is linearly independent from previously found elements. Furthermore, $b_1(S)$ is bounded above by the number of edges of $\calT$, which is at most $3|\calT|$. The number of arc types is bounded above $3|\calT|$. Hence, the number of 0-handles of $S \cut C$ that do not lie between parallel normal arcs of $C$ is at most $4|\calT|$. Each component of $S \cut C$ contains at least one such 0-handle, as we are assuming that no two components of $C$ are normally parallel. Therefore, $|S \cut C|$ is at most $4|\calT|$. We deduce that $|C|$ is bounded above by $7|\calT|$.
%\end{proof}

We now consider how much information is required to specify a pairing system giving a 1-manifold $C$. The method for giving a pairing system is explained in Section \ref{Sec:Intro}. Each 1-cell must be oriented. The number of 1-cells is $6|\calH^1| \leq 6||\calH||$. The number of intersection points between each 1-cell and $C$ must be given. The total number of intersection points is $2w(C) + |\partial C|$, since each component of $C \cap \calH^1$ contributes $2$ to this sum, and each component of $\partial C$ contributes $1$. Hence, the maximal number of intersection points with any 1-cell lies between $(2w(C) + |\partial C|)/6||\calH||$ and $2w(C) + |\partial C|$. So, to give these numbers of intersection points requires of the order of $||\calH|| \log (w(C) + |\partial C|)$ bits of data. Finally, there is a pairing for each arc type, and the number of arc types lies between $\max \{ 1, \mathrm{simp}(C) \}$ and $9||\calH|| + 2\mathrm{simp}(C)$, by Lemma \ref{Lem:NumberOfPairingsCurve}. Hence, the total amount of data is bounded above and below by linear functions of $(||\calH|| + \mathrm{simp}(C)) \log (w(C) + |\partial C|)$.

\subsection{A vector representation of normal curves and surfaces}
\label{Sec:Vector}

In this subsection, we present an alternative way of specifying a standard 1-manifold $C$ in a surface $S$ with a handle structure $\calH$.
Let $\calX$ be the cell structure for $S$ that is associated with $\calH$. We subdivide $\calX$ by
adding vertices to the interior of the 1-cells, one in the boundary of each simplifying disc, disjoint from $C$. After this has been done, each arc of intersection
between $C$ and each 2-cell joins distinct edges of the 2-cell. Thus, we may encode $C$ as follows. For each 2-cell
and for each pair of distinct edges of that 2-cell, we simply count the number of arcs of $C$ running between these
two edges. This list of non-negative integers is the \emph{vector} $(C)$. Thus, we can specify $C$ simply by 
specifying the subdivision of the 1-cells of the cell structure of $\calH$ and by giving the vector $(C)$.

Note that the amount of information required to specify a standard 1-manifold $C$ in this way is at most a polynomial function of $\mathrm{simp}(C)$, $||\calH||$, and $\log (w(C) + |\partial C|)$, 
for the following reason. Before the subdivision of the cell structure, the total number of 1-cells was $6 |\calH^1|$ by Lemma \ref{Lem:Number1Cells}.
The number of subdivisions that we perform is at most $\mathrm{simp}(C)$, and each such subdivision increases the number of 1-cells by
$1$. So afterwards, this number is at most $6|\calH^1| + \mathrm{simp}(C)$. Thus, the number of 1-cells in the boundary of the 2-cells is at most $12|\calH^1| + 2\mathrm{simp}(C)$.
So the number of arc types in the 2-cells is at most $(12 |\calH^1| + 2 \mathrm{simp}(C))^2$. For each arc type, we specify the number of arcs of this type,
and the total number of arcs is at most $2 w(C) + |\partial C|/2$, because each arc of $C$ alternates between 0-handles and 1-handles of $\calH$,
and the number of times it runs over the 1-handles is its weight. Note that the addition of $|\partial C|/2$ is required to deal with arc components of $C$, each 
of which has one more component of intersection with the 0-handles than with the 1-handles. 
Hence, the number of digits required to specify each co-ordinate of $(C)$ is at most $\log (2 w(C) + |\partial C|/2) $.

\section{The algorithm of Agol-Hass-Thurston}
\label{Sec:AHT}

A crucial tool that is used in this paper is the algorithm of Agol, Hass and Thurston \cite{AHT}. We describe it in this section. The original use of the algorithm
was to determine the number of components of a normal 1-manifold in a triangulated surface, and, in one higher dimension, the number of components
of a normal surface in a triangulated 3-manifold. However, it is a very general algorithm with applications in many different situations in low-dimensional
topology. (See, for example, \cite{Lackenby:Efficient} where it is used in several ways.)

A \emph{pairing} is an order-preserving or order-reversing bijection $[a,b] \cap \mathbb{Z} \rightarrow [c,d] \cap \mathbb{Z}$ for integers $a \leq b$ and $c \leq d$. For an integer $x$ and pairing $g$, we only write $g(x)$ if $x$ is in the domain of $g$. Similarly, we only write $g^{-1}(x)$ if $x$ is in the domain of $g^{-1}$.
The input to the algorithm is:
\begin{enumerate}
\item non-negative integers $N$ and $r$;
\item a collection of $r$ pairings $g_i \colon [a_i, b_i] \cap \mathbb{Z} \rightarrow [c_i, d_i] \cap \mathbb{Z}$, where $1 \leq a_i \leq b_i \leq N$ and $1 \leq c_i \leq d_i \leq N$.
\end{enumerate}
Two integers $x$ and $y$ are said to be \emph{in the same orbit} of these pairings if there is a sequence $i_1, \dots, i_k$ of integers between $1$ and $r$ and integers $\varepsilon_1, \dots, \varepsilon_k \in \{-1, 1\}$ such that $g_{i_k}^{\varepsilon_k} \dots g_{i_1}^{\varepsilon_1}(x) = y$. This forms an equivalent relation, and the equivalence classes are called \emph{orbits}.

The following is the basic algorithm that Agol, Hass and Thurston provided. They also gave an enhanced version that we will describe later.

\begin{theorem}[{\sc Basic AHT}]
There is an algorithm that takes, as its input, non-negative integers $N$ and $r$ and a collection of pairings as above, and provides as its output the number of orbits. The running time is bounded above by a polynomial function of $r$ and $\log N$.
\end{theorem}

As an immediate consequence of this and Lemma \ref{Lem:NumberOfPairingsCurve}, we obtain the following, which is a minor extension of a result of Agol, Hass and Thurston \cite[Corollary 13]{AHT} from normal 1-manifolds to standard 1-manifolds. The crucial observation is that a standard 1-manifold comes equipped with a pairing system, and orbits of this system correspond to components of the 1-manifold.

\begin{theorem}[{\sc Number of components of standard 1-manifold}]
\label{Thm:NumberComponentsCurves}
There is an algorithm that takes, as its input, a triangulation $\calT$ or handle structure $\calH$ for a compact surface $S$ and a standard 1-manifold $C$, and provides as its output the number of components of $C$. The running time is bounded above by a polynomial function of $\mathrm{simp}(C)$, $\log (w(C) + |\partial C|)$ and $|\calT|$ or $||\calH||$.
\end{theorem}

\begin{proof} 
The standard 1-manifold $C$ is specified using a pairing system. For each edge $e$ of $\calT$ or each 1-cell $e$ of the cell structure for $\calH$, the quantity $|e \cap C|$ is given as part of the pairing system. Let $N$ equal $\sum_e |e \cap C|$. In the case when we are given a triangulation, $N$ is just the weight of $C$. In the case of a handle structure, $N$ is $|\partial C| + 2 w(C)$. If we pick a total ordering on the edges, then, using the given orientation on the edges, we can identify $[1,N] \cap \mathbb{Z}$ with the points of intersection between $C$ and the edges $e$. Thus we obtain pairings between sub-intervals of $[1,N] \cap \mathbb{Z}$. The number $r$ of these pairings is at most $9||\calH|| + 2 \mathrm{simp}(C)$ in the case where we have a handle structure $\calH$ and is similarly bounded when we have a triangulation $\calT$. Then {\sc Basic AHT} counts the number of orbits, which is the number of components of $C$. Its running time is bounded above by a polynomial function of $\mathrm{simp}(C)$, $\log (w(C) + |\partial C|)$ and $|\calT|$ or $||\calH||$.
\end{proof}

There is an enhanced version of the Agol-Hass-Thurston algorithm, which is as follows. Its input is the non-negative integers $N$ and $r$, the pairings $g_i$ and also a function $\phi \colon [1,N] \cap \mathbb{Z} \rightarrow \mathbb{Z}^m$ for some positive integer $m$. The \emph{$\phi$-value} of an orbit is then the sum of $\phi(x)$, over all elements $x$ in the orbit. One measure of complexity of $\phi$ is $M = \max \{\ |\phi(x)|_1+1 : x \in [1,N] \cap \mathbb{Z} \}$, where $| \cdot |_1$ is the $\ell^1$ norm. Another measure is the number $s$ of integers $x$ in $[2,N]$ for which $\phi(x) \not= \phi(x-1)$. 

\begin{theorem}[{\sc AHT}]
There is an algorithm that takes as its input non-negative integers $N$ and $r$, a collection of $r$ pairings as above, and a function $\phi \colon [1,N] \cap \mathbb{Z} \rightarrow \mathbb{Z}^m$ as above, and provides as its output a list of orbits together with their $\phi$-values. The running time is bounded above by a polynomial function of $r$, $s$, $m$, $\log N$ and $\log M$, where $s$ and $M$ are as defined above.
\end{theorem}

This has many interesting applications. We present one of them now, which strengthens Theorem \ref{Thm:NumberComponentsCurves}.

\begin{theorem}[{\sc Components of standard 1-manifold}]
\label{Thm:ComponentsCurves}
There is an algorithm that takes, as its input, a triangulation $\calT$ or handle structure $\calH$ for a compact surface $S$ and a standard 1-manifold $C$, and provides as its output a list of the components of $C$, given as vectors. The running time is bounded above by a polynomial function of $\mathrm{simp}(C)$, $\log (w(C) + |\partial C|)$ and $|\calT|$ or $||\calH||$. 
\end{theorem}

\begin{proof} The 1-manifold $C$ can be given as a pairing system or as a vector $(C)$ as described in Section \ref{Sec:Vector}. If we are given it as a pairing system, we can convert it to a vector $(C)$. Similarly, if we are given $(C)$, we can convert it to a pairing system. We will use both such representations of $C$.

Say that there are $m$ arc types of $C$ in the 2-cells. For each arc type, we will define a function $\phi_i \colon [1,N] \cap \mathbb{Z} \rightarrow \mathbb{Z}$. These will combine to form the function $\phi \colon [1,N] \cap \mathbb{Z} \rightarrow \mathbb{Z}^m$. In other words, the composition of $\phi$ with projection onto the $i$th co-ordinate of $\mathbb{Z}^m$ is $\phi_i$. For the given arc type, this determines a pairing $[a,b] \cap \mathbb{Z} \rightarrow [c,d] \cap \mathbb{Z}$. Define $\phi_i$ to be $1$ on the domain of this pairing and $0$ elsewhere. Then if $C'$ is a component of $C$, then $C'$ corresponds to a subset of $[1,N]$ and if we apply $\phi$ to this subset, then the resulting integer is the co-ordinate of $(C')$ corresponding to this arc type. Note that number $s$ of values of $x$ for which $\phi(x) \not= \phi(x-1)$ is at most twice the number of arc types.

Hence, when we apply {\sc AHT} to this pairing system with this function $\phi$, then the output is a list of the components of $C$, and the $\phi$-value of each component $C'$ gives its vector $(C')$. \end{proof}

\section{Cutting along a standard 1-manifold}
\label{Sec:Cut}

One of the advantages of using handle structures rather than triangulations is that when we cut along a standard 1-manifold, the resulting manifold inherits a handle structure, as follows.

\begin{definition}
\label{Def:InheritedHS}
Let $S$ be a compact surface with a handle structure $\calH$. Let $C$ be a standard 1-manifold properly embedded in $S$. Then the \emph{inherited handle structure} on $S \cut C$ has an $i$-handle for each component of $\calH^i \cut C$. In the case where $S$ is instead given as a triangulation $\calT$, we let $\calH$ be the handle structure dual to $\calT$ and then the inherited handle structure on $S \cut C$ is defined as above.
\end{definition}

The inherited handle structure on $S \cut C$ is frequently not practical to use because it can have a large number of handles. Specifically, its number of 1-handles is at least $w(C)$. However, when $w(C)$ is large, many of the handles of $S \cut C$ have the following form.

\begin{definition}
A \emph{parallelity handle} for $S \cut C$ is a handle in the inherited handle structure that lies between two adjacent normally parallel arcs of $C$ in a handle $H$ of $\calH$. Thus, its boundary is the concatenation of four arcs. Two of these are components of $C \cap H$, and the other two are arcs in $\partial S \cap H$ or $H \cap \calH^1 \cap \calH^0$. The parallelity handle is an $I$-bundle over an interval, with the $(\partial I)$-bundle being its intersection with $C$.
\end{definition}

\begin{definition}
The \emph{parallelity bundle} for $S \cut C$ is the union of its parallelity handles. It is an $I$-bundle over a 1-manifold, with the $(\partial I)$-bundle being its intersection with $C$.
\end{definition}

In general, for an $I$-bundle $B$, its \emph{horizontal boundary} $\partial_h B$ is the $(\partial I)$-bundle. Its \emph{vertical boundary} $\partial_v B$ is $\partial B \cut \partial_h B$.

Note that two distinct components $C_1$ and $C_2$ of $C$ are normally parallel precisely when some component of $S \cut C$ is a component $B$ of the parallelity bundle with $\partial_h B = C_1 \cup C_2$.

%\begin{definition}
%An \emph{interior parallelity 0-handle} for $S \cut C$ is a 0-handle $H$ such that $\partial H$ is a concatenation of an arc in $C$, an arc in $\calH^1$, an arc in $C$ and an arc in $\calH^1$. A \emph{parallelity 1-handle} for $S \cut C$ is a 1-handle $H$ that intersects $C$ in two arcs. The \emph{interior parallelity bundle} for $S \cut C$ is the union of the interior parallelity 0-handles and the parallelity 1-handles.
%\end{definition}

%\begin{definition} A \emph{parallelity 0-handle} for $S \cut C$ is either an interior parallelity 0-handle or a 0-handle $H$ such that $\partial H$ is a concatenation of an arc in $C$, an arc in $\calH^1$, an arc in $C$ and an arc in $\partial S \cut C$. The \emph{parallelity bundle} for $S \cut C$ is the union of the parallelity 0-handles and parallelity 1-handles.
%\end{definition}

%\begin{definition} 
%A \emph{semi-parallelity 0-handle} for $S \cut C$ is a 0-handle $H$ such that $\partial H$ is a concatenation of an arc in $C$, an arc in $\calH^1$, an arc in $C$ and an arc in $\partial S \cut C$, an arc in $\calH^1$. A \emph{semi-parallelity 1-handle} for $S \cut C$ is a 1-handle that intersects $C$ in one arc. The \emph{semi-parallelity bundle} is the union of the semi-parallelity handles.
%\end{definition}

\begin{figure}[h]
\centering
\includegraphics[width=0.5\textwidth]{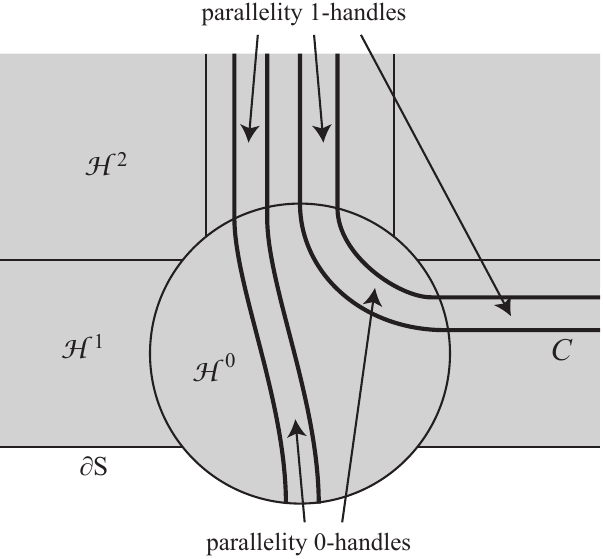}
\caption{Parallelity handles of $S \cut C$}
\label{Fig:ParHandles}
\end{figure}

\begin{definition}
Let $S$ be a compact surface with a handle structure $\calH$. Let $C$ be a standard 1-manifold properly embedded in $S$. Then the \emph{reduced handle structure} $\calH'$  is defined as follows. We start with the handle structure on $S \cut C$ that it inherits. Let $\calB$ be its parallelity bundle. Then $\calH'$ is obtained by removing all the handles in $\calB$ and replacing each component of $\calB$ that misses $\partial S$ and is an $I$-bundle over an interval by a single 1-handle. Each component of $S \cut C$ that is a union of parallelity handles is removed. 
\end{definition}

\begin{remark}
\label{Rem:ReducedHS}
As explained in the definition, components of $S \cut C$ that are a union of parallelity handles are removed. Thus, the reduced handle structure might be not a handle structure for all of $S \cut C$, but only certain components of $S \cut C$. Moreover each component of $\calB$ that intersects $\partial S$ in a single arc is removed, but this does not change the homeomorphism type of the surface.
\end{remark}

Note that by Lemma \ref{Lem:NumberOfPairingsCurve}, the number of handles of $S \cut C$ that are not parallelity handles is at most $O(||\calH|| + \mathrm{simp}(C))$. The remaining handles of the reduced handle structure are components $B$ of the $\calB$ that are disjoint from $\partial S$ and that are an $I$-bundle over an interval. For each such component $B$, $\partial_v B$ is incident to arcs of $C \cap \calH^0$ that are outermost in their parallelism classes. Hence the number of such components of $\calB$ is also at most $O(||\calH|| + \mathrm{simp}(C))$. Therefore,
the number of handles in the reduced handle structure is at most $O(||\calH|| + \mathrm{simp}(C))$.

\begin{theorem}[{\sc Cut along standard 1-manifold}]
\label{Thm:CutAlong1Manifold}
There is an algorithm that takes, as its input, a handle structure $\calH$ for a compact surface $S$ and a standard 1-manifold $C$ properly embedded in $S$, and outputs the reduced handle structure on the relevant components of $S \cut C$. The running time is at most a polynomial function of $||\calH||$, $\mathrm{simp}(C)$, and $\log (w(C) + |\partial C|)$.
\end{theorem}

\begin{proof} A closely related algorithm is given in \cite[Theorems 9.2 and 9.3]{Lackenby:Efficient} and so we just sketch it here.

Let $\calB$ be the parallelity bundle of $S \cut C$. The first stage in the procedure is to determine the handle structure on $(S \cut C) \cut \calB$, as follows. For each handle $H$ of $\calH$, it is divided into handles $H \cut C$. Those lying in $\calB$ are the ones lying between parallel arcs of $H \cap C$. Thus, we can form $(H \cut C) \cut \calB$ by replacing each arc type of $H \cap C$ by a single arc, and then cutting along the resulting arcs. In this way, we can form the handle structure on $(S \cut C) \cut \calB$, and also identify its components of intersection with $\partial_v \calB$.

Say that there are $k$ arc types of $C$ in the 0-handles and 1-handles of $\calH$. For $1 \leq i \leq k$, say there are $N_i$ arcs of that type. Then there are $\max \{ 0, N_i -1 \}$ parallelity handles of $S \cut C$ between these arcs. Let $N = 2 \sum \max \{ 0, N_i -1 \}$, which is the total number of components of $\partial_h H'$, as $H'$ runs over all parallelity handles. This will be one of the inputs into the {\sc AHT} algorithm. 

Let $m$ be the number of boundary components of $\partial_v \calB$ disjoint from $\partial S$. This will be another input to {\sc AHT}. Note that these components of $\partial (\partial_v \calB)$ lie in $(S \cut C) \cut \calB$, and hence $m$ is bounded above by a linear function of $||\calH||$ and $\mathrm{simp}(C)$. We will define a function $\phi \colon [1,N] \rightarrow \mathbb{Z}^m$. For each integer between $1$ and $m$, we identify it with a specific boundary component of $\partial_v \calB \cut \partial S$.
Consider a 1-handle $H_1$ of $\calH$ attached onto a 0-handle $H_0$. Suppose that $H_1 \cap C$ consists of more than 1 arc. Then between these arcs lie handles of $\calB$. Each such handle $H_1'$ has two components of $\partial_h H'_1$, each of which is identified with an integer $i$ between $1$ and $N$. The handle $H_1'$ is incident to a handle of $H_0 \cut C$ that may or may not be parallelity. If it is not a parallelity handle, then each component of $\partial_h H_1'$ meets a component of $\partial( \partial_v \calB)$, the $j$th component say. Then we define $\phi(i)$ to be $1$ at the $j$ co-ordinate and $0$ at the remaining co-ordinates. For the remaining handles $H'_1$ that are incident to parallelity 0-handles in $H_0$, this incidence defines pairings, each going from a subset of $[1,N]$ to another subset of $[1,N]$.

When we run {\sc AHT}, the output is a collection of orbits. Each orbit corresponds to a component of $\partial_h \calB$. For each such component $\alpha$, we examine its $\phi$-value to determine the location of $\partial \alpha \cut \partial S$. The number of components of $\partial \alpha \cut \partial S$ is $0$, $1$ or $2$. If it is $2$, then the component of $\calB$ containing $\alpha$ is an $I$-bundle over an interval and is disjoint from $\partial S$, and hence forms a 1-handle in the reduced handle structure. On the other hand, if this number is less than $2$, then the component of $\calB$ containing $\alpha$ is removed when forming the reduced handle structure.

We can therefore compute which components of $\calB$ form 1-handles of the reduced handle structure and how they are attached onto $(S \cut C) \cut \calB$.
\end{proof}

\begin{remark} 
\label{Rem:ReducedHSEmbeds}
It is worth stating that the reduced handle structure $\calH'$ on $S \cut C$ is provided not just as an abstract handle structure but also with information about how it embeds in $\calH$. More precisely, for each handle of $\calH'$ that did not come from a component of the parallelity bundle $\calB$, we have an identification between it and the handle of $S \cut C$ that it came from. We can also determine the location of the handles arising from components of $\calB$, in the following sense.
\end{remark}

\begin{theorem}[{\sc Location of parallelity bundle component}]
\label{Thm:LocationParBundle}
There is an algorithm that takes, as its input, the following:
\begin{enumerate}
\item a handle structure $\calH$ for a compact surface $S$;
\item a standard 1-manifold $C$ properly embedded in $S$;
\item a point $p$ in the intersection between $S - C$ and a 1-cell of the associated cell structure on $S$;
\end{enumerate} and provides the following:
\begin{enumerate}
\item the vector(s) in $S$ for the component(s) of $\partial_hB$, where $B$ is the component of the parallelity bundle for $S \cut C$ containing $p$ (or  the zero vector if $p$ does not lie in the parallelity bundle);
\item for each such component $\alpha$ of $\partial_h B$, the location of $\partial \alpha - \partial S$.
\end{enumerate}
The running time is at most a polynomial function of $||\calH||$, $\mathrm{simp}(C)$, and $\log (w(C) + |\partial C|)$.
\end{theorem}

\begin{proof} In the proof of Theorem \ref{Thm:CutAlong1Manifold}, $N$ was the number of horizontal boundary components of the parallelity handles for $S \cut C$. Each such handle has an associated arc type of $C$. Say that there are $L$ arc types. Then instead of defining $\phi$ to have image in $\mathbb{Z}^m$, we define it to have image in $\mathbb{Z}^{m+L+1}$. The first $m$ co-ordinates are as in the proof of Theorem \ref{Thm:CutAlong1Manifold}. The next $L$ co-ordinates correspond to the $L$ arc types and record the arc type of a parallelity handle. The final co-ordinate is defined to be $1$ if the handle contains $p$, and zero otherwise. The output of {\sc AHT} is a list of orbits, each of which corresponds to a component of $\partial_h \calB$. For each such component $\alpha$ of $\partial_h \calB$, we can determine whether it lies in the same component of $\calB$ as $p$ by reading off the final component of the $\phi$-value of $\alpha$. We can read off the vector of $\alpha$ by examining the previous $L$ co-ordinates.
\end{proof}

It will be useful to consider the following close relative of the parallelity bundle.

\begin{definition} 
A \emph{semi-parallelity handle} for $S \cut C$ is a handle $H$ of its inherited handle structure such that $\partial H$ is a concatenation of an arc in $C$, an arc in $\calH^0 \cap \calH^1$, an arc in $\partial S$ and an arc in $\calH^0 \cap \calH^1$. It is an $I$-bundle, where the $(\partial I)$-bundle is its intersection with $C \cup \partial S$. The \emph{semi-parallelity bundle} is the union of the semi-parallelity handles. It too is an $I$-bundle, where the $(\partial I)$-bundle is its intersection with $C \cup \partial S$.
\end{definition}

Note that it is straightforward to compute the location of the semi-parallelity bundle, because the number of semi-parallelity handles is at most twice the number of components of $\partial S \cap \calH^1$, and this is at most $2|\calH^1| \leq 2 ||\calH||$.

\section{Fast normalisation of standard 1-manifolds}
\label{Sec:Normalise}

In this section, we prove Theorem \ref{Thm:Normalise1Manifold}. The proof of Theorem \ref{Thm:NormaliseStandardCurves} is entirely analogous and is omitted.

\begin{named}{Theorem \ref{Thm:Normalise1Manifold}}[{\sc Normalise a standard 1-manifold in a handle structure}]
There is an algorithm that takes as its input the following:
\begin{enumerate}
\item a handle structure $\calH$ for a compact surface $S$;
\item a standard 1-manifold $C$ properly embedded in $S$;
\end{enumerate}
and outputs a normal 1-manifold $\overline{C}$ that is obtained from $C$ by an isotopy plus possibly discarding some curves that bound discs and boundary-parallel arcs. The running time for the algorithm is bounded above by a polynomial function of $||\calH||$, $\mathrm{simp}(C)$ and $\log (w(C) + |\partial C|)$.
\end{named}

We are given a handle structure $\calH$ for a compact surface $S$ and a standard 1-manifold $C$ properly embedded in $S$.

We first note that we may assume that no two components of $C$ are normally parallel, via the following procedure. We may compute a list of components of $C$ by {\sc Components of standard 1-manifold} (Theorem \ref{Thm:ComponentsCurves}). We may then replace $C$ by $C_-$, which has a single copy of each normal parallelism type of the components of $C$. Suppose that we can isotope $C_-$ to a normal 1-manifold $\overline{C}_-$. Then we can place $C$ into normal form by replacing each component of $\overline{C}_-$ by a suitable number of copies. So, we will now assume that no two components of $C$ are normally parallel. This is useful because in this situation, the reduced handle structure is a handle structure for a surface homeomorphic to $S \cut C$, except possibly with some M\"obius band components removed. (See Remark \ref{Rem:ReducedHS}.)

Suppose that $C$ is not normal. Then it has a simplifying disc $D$, by Lemma \ref{Lem:NormalNoSimplifying}. If $D$ is disjoint from $\calH^1$, then $D \cap C$ is a boundary-parallel arc, and we can remove this arc and all normally parallel copies of it. Suppose therefore that $D \cap C$ intersects $\calH^1$. By definition of a simplifying disc, this intersection is an arc. We may then isotope $C$ across $D$ and across the incident 1-handle. The resulting 1-manifold $C'$ has smaller weight than $C$. It might not in fact be standard, because it might have a component of intersection with a 0-handle that is a simple closed curve. But in that case, we can remove this component. So, we may assume that $C'$ is standard. Hence, by induction on the weight of $C$, we can normalise it. This is the usual normalisation process.

However, this does not obviously lead to an efficient algorithm, for two reasons. Firstly, the weight of $C'$ is only $w(C)-1$ or $w(C)-2$. Hence, there may need to be of the order of $w(C)$ such isotopies before $C$ is normalised, whereas we require the algorithm to terminate in polynomial time as a function of $\log (w(C) + |\partial C|)$. Secondly, we have seen in Section \ref{SubSec:SizePairing} that the amount of information required to describe a standard curve $C$ is of the order $(\mathrm{simp}(C) + ||\calH||) \log (w(C) + |\partial C|)$. Thus, to prevent this quantity increasing without bound, we would need to control $\mathrm{simp}(C)$. 

Instead of using $w(C)$ as our measure of progress through the algorithm that we will present, we use $\mathrm{simp}(C)$. However $\mathrm{simp}(C')$ might actually be greater than $\mathrm{simp}(C)$. An example is shown in Figure \ref{Fig:IncreaseSimplifyingIndex}. 

\begin{figure}[h]
\centering
\includegraphics[width=0.7\textwidth]{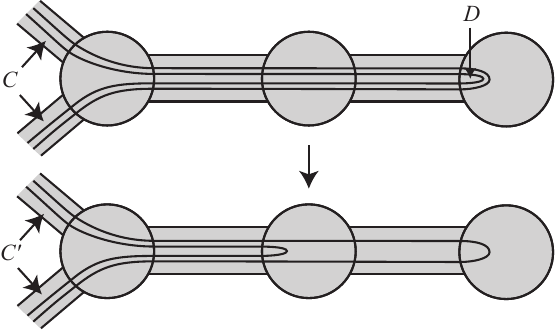}
\caption{Shown is a standard curve $C$ with a simplifying disc $D$. When an isotopy is performed across $D$ and the incident 1-handle, the effect is to increase the simplification number}
\label{Fig:IncreaseSimplifyingIndex}
\end{figure}

In general, this happens as follows. For simplicity, we initially assume that $S$ is closed. Let $H_0$ be the 0-handle of $\calH$ containing $D$. Running parallel to $D \cap C$, there might be another arc of $H_0 \cap C$. This would be part of the boundary of a new simplifying disc for $C'$. But also the new arc of $\calH^0 \cap C'$ created by the isotopy across $D$ and the incident 1-handle may also bound a new simplifying disc $D'$. There are two possibilities for $D'$: either the interior of $D'$ intersects the disc region between $C$ and $C'$ or it is disjoint from it. These two cases are shown in Figure \ref{Fig:NewSimplifyingDisc}. 

\begin{figure}[h]
\centering
\includegraphics[width=0.7\textwidth]{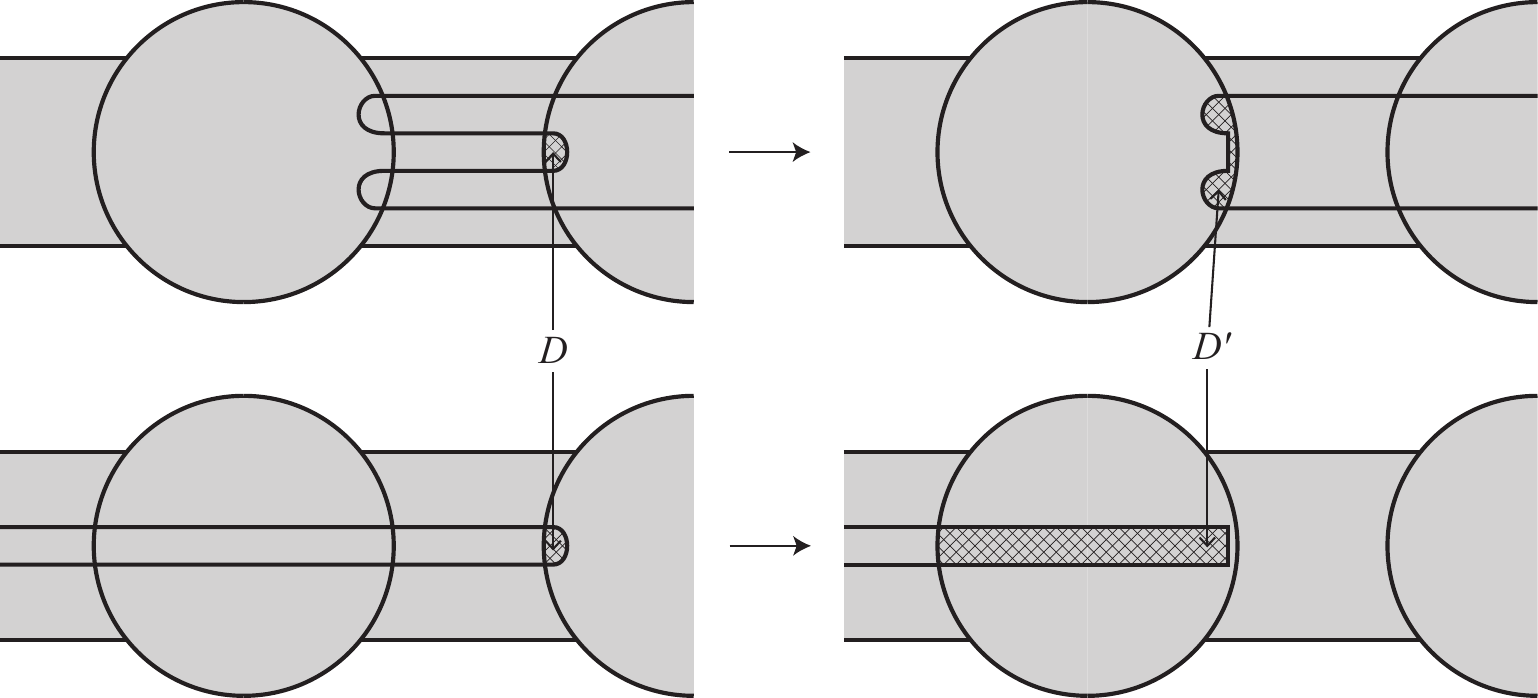} 
\caption{Shown on the left are simplifying discs $D$ for the standard curve $C$. When the associated isotopy of $C$ is performed, new simplifying discs $D'$ are created.}
\label{Fig:NewSimplifyingDisc}
\end{figure}

In the former case, then $D'$ was actually formed from two simplifying discs of $C$, both of which are removed in the isotopy. Thus, in this case, it is possible to deduce that $\mathrm{simp}(C') <  \mathrm{simp}(C)$. So suppose that the interior of $D'$ is disjoint from the region between $C$ and $C'$. Then $D'$ lay in a parallelity handle of $S \cut C$. Thus, we are led to the following definition. 

\begin{definition}
\label{Def:IsotopyRegionClosed}
Let $\calH$ be a handle structure for a closed surface $S$. Let $C$ be a standard 1-manifold properly embedded in $S$. Let $D$ be a simplifying disc for $C$. Then the \emph{isotopy region} associated with $D$ is the union of the following subsets of $S \cut C$:
\begin{enumerate}
\item the simplifying disc $D$;
\item the component $B$ of the parallelity bundle for $S \cut C$ that is incident to $D$;
\item a small regular neighbourhood of $\partial_v B$ in $S \cut C$. 
\end{enumerate}
(See Figure \ref{Fig:IsotopyRegion}.) We say that the \emph{horizontal boundary} of $R$ is $\partial R \cap C \cut \partial D$.
\end{definition}

\begin{figure}[h]
\centering
\includegraphics[width=0.6\textwidth]{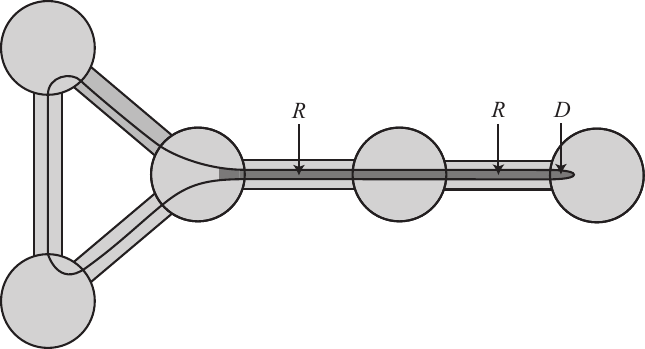}
\caption{The isotopy region $R$ associated with the simplifying disc $D$ is shaded}
\label{Fig:IsotopyRegion}
\end{figure}

The version of the above definition when $S$ has boundary is a bit more complicated. It involves the following type of 0-handle of $S \cut C$.

\begin{definition}
Let $\calH$ be a handle structure for a compact surface $S$. Let $C$ be a standard 1-manifold properly embedded in $S$. A 0-handle of $S \cut C$ is an \emph{interpolating 0-handle} if its boundary is a concatenation of an arc of $C \cap \calH^0$, an arc in $\calH^0 \cap \calH^1$, an another arc of $C \cap \calH^0$, another arc in $\calH^0 \cap \calH^1$ and an arc in $\partial S$.
See Figure \ref{Fig:IsotopyRegionBoundary}.
\end{definition}

\begin{definition}
\label{Def:IsotopyRegion}
Let $\calH$ be a handle structure for a compact surface $S$ with non-empty boundary. 
Let $C$ be a standard 1-manifold properly embedded in $S$. Let $D$ be a simplifying disc for $C$. Then the \emph{isotopy region} $R$ associated with $D$ is the union of the following subsets of $S \cut C$:
\begin{enumerate}
\item the simplifying disc $D$;
\item the component $B_1$ of the parallelity or semi-parallelity bundle that is incident to $D$;
\item if $D$ is an interior-simplifying disc and $B_1$ is incident to an interpolating 0-handle of $S \cut C$, then this 0-handle is also part of $R$ and so is the incident component $B_2$ of the semi-parallelity bundle;
\item a small regular neighbourhood of $\partial_v B_1$ in $S \cut C$, as well as a small regular neighbourhood of $\partial_v B_2$ if $B_2$ is defined. 
\end{enumerate}
(See Figure \ref{Fig:IsotopyRegionBoundary}.) We say that the \emph{horizontal boundary} of $R$ is $\partial R \cap (C \cup \partial S) \cut \partial D$.
\end{definition}

\begin{figure}[h]
\centering
\includegraphics[width=0.6\textwidth]{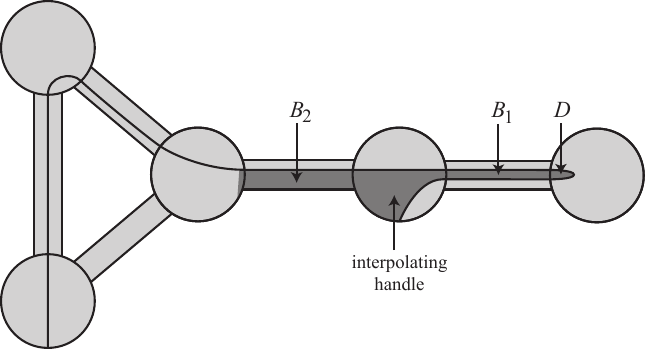}
\caption{The isotopy region associated with a simplifying disc $D$ in the case where $\partial S \not= \emptyset$.}
\label{Fig:IsotopyRegionBoundary}
\end{figure}

The reason why an interpolating handle might be included in the isotopy region is as follows. Suppose that $C$ consists of two parallel copies of the 1-manifold shown in Figure \ref{Fig:IsotopyRegionBoundary}. If we were to isotope one of these copies across $D$ and $B_1$ but no further, then the resulting 1-manifold $C'$ would have higher simplification number than $C$. Although the interior-simplifying disc $D$ would have been removed, an interior-simplifying disc and a boundary-simplifying disc would have been created.

It is also worth pointing out that here the situation for triangulations is somewhat simpler than for handle structures. In the case where $S$ comes with a triangulation, then Definition \ref{Def:IsotopyRegionClosed} is used rather than Definition \ref{Def:IsotopyRegion}. In particular, there is no analogue of an interpolation handle.

The following key lemma states that the simplification number does not go up when an isotopy is performed across an isotopy region.

\begin{lemma}
\label{Lem:SimpDoesNotGoUp}
Let $\calH$ be a handle structure for a compact surface $S$, and let $C$ be a standard properly embedded 1-manifold.
Let $C'$ be the 1-manifold obtained by isotoping $C$ across the isotopy region associated with a simplifying disc $D$. Then $\mathrm{simp}(C') \leq  \mathrm{simp}(C)$ and $w(C') < w(C)$. Furthermore, if $\mathrm{simp}(C') =  \mathrm{simp}(C)$, then there is an arc of $\calH^0 \cap C'$ that is normally parallel to $D \cap C$, and all the components of $C \cap \calH^0$ between these two arcs are normally parallel to $D \cap C$.
\end{lemma}

\begin{proof}
We define a function
$$f \colon \{ \text{simplifying discs for } C' \} \rightarrow \mathbb{P} \{ \text{simplifying discs for } C \}$$
as follows. For a simplifying disc $D'$ for $C'$, we set $f(D')$ to be the set of simplifying discs for $C$ that lie in $D'$. We will establish the following properties:
\begin{enumerate}
\item $f(D') \not= \emptyset$ for every simplifying disc $D'$ for $C'$;
\item $f(D'_1) \cap f(D'_2) = \emptyset$ for distinct simplifying discs $D'_1$ and $D'_2$ for $C'$;
\item if $D'$ is an interior-simplifying disc for $C'$, then $f(D')$ consists of interior-simplifying discs;
\item if $f(D') = \{ D_1 \}$ for a single disc $D_1$ that is interior-simplifying exactly when $D'$ is interior-simplifying, then $C\cap \mathrm{int}(D')$ consists of a (possibly empty) collection of parallel arcs, all normally parallel to $D' \cap C'$.
\end{enumerate}
These properties quickly imply the lemma, as follows. Via (1), each simplifying disc for $C'$ gives rise to a non-empty collection of simplifying discs for $C$. As a result of (2), these sets are pairwise disjoint. Hence, the number of simplifying discs for $C$ is at least the number of simplifying discs for $C'$. Moreover, by (3), each interior-simplifying disc for $C'$ gives rise to at least one interior-simplifying disc for $C$. So we deduce that $\mathrm{simp}(C') \leq  \mathrm{simp}(C)$. In the case of equality, each interior-simplifying disc for $C'$ contains exactly one simplifying disc for $C$, which is interior-simplifying; and each boundary-simplifying disc for $C'$ contains exactly one simplifying disc for $C$, which is boundary-simplifying; and each simplifying disc for $C$ lies in a simplifying disc for $C'$. (It is here that we use the fact that interior-simplifying discs and boundary-simplifying discs for $C$ have different contributions to $\mathrm{simp}(C)$.) In particular, the simplifying disc $D$ must lie in a simplifying disc $D'$ for $C'$. Since $|f(D')| = 1$ and the hypotheses of (4) are satisfied, then (4) implies that $C\cap \mathrm{int}(D')$ consists of a (possibly empty) collection of parallel arcs, all normally parallel to $D' \cap C'$. In fact, these arcs are non-empty in this case, because otherwise $D' = D$, but $D$ is removed by the isotopy. Thus, we deduce that, in the case where $\mathrm{simp}(C') = \mathrm{simp}(C)$, there is an arc of $\calH^0 \cap C'$ that is normally parallel to $D \cap C$, and all the components of $C \cap \calH^0$ between these two arcs are normally parallel to $D \cap C$, as required.

We now prove (1) - (4). Properties (2) and (3) are trivial. For (2), note that distinct simplifying discs $D'_1$ and $D'_2$ for $C'$ are disjoint, and therefore the simplifying discs for $C$ that lie within them are disjoint. For (3), note that an interior-simplifying disc $D'$ is disjoint from $\partial S$ and hence so is every disc that it contains. 

For (4), observe that if the arcs $C\cap \mathrm{int}(D')$ are non-empty and not all parallel in $D'$, then at least two separate off sub-discs of $D'$ that are disjoint from $\partial D' \cap C'$. These discs are then simplifying discs, implying that $|f(D')| > 1$. So if $f(D') = \{ D_1 \}$, then the arcs $C \cap \mathrm{int}(D')$ are all parallel in $D'$. Moreover when $D_1$ and $D'$ are both interior-simplifying or both boundary-simplifying, then the arcs $C \cap \mathrm{int}(D')$ are all normally parallel to $D' \cap C'$, as required.

For a simplifying disc $D'$ for $C'$ that is also a simplifying disc for $C$, property (1) is immediate. Therefore,
what concerns us are simplifying discs for $C'$ that are not simplifying discs for $C$. Let $D'$ be one such disc. Let $\alpha'$ be the arc $\partial D' \cap C'$. We consider two possibilities: either $\alpha'$ is a subset of $C$ or it is not.

Suppose first that $\alpha'$ is a subset of $C$. Then the reason that $D'$ is not a simplifying disc for $C$ must have been because $C$ intersected the interior of $D'$ and the arcs of intersection between $C$ and the interior of $D'$ were removed in the isotopy. In that case, at least one of these arcs bounded a simplifying disc for $C$ that was removed in the isotopy. 
This establishes (1) in this case.

Now consider the case where $\alpha'$ is not a subset of $C$. It therefore contains the arc $C' - C$. Thus, the isotopy region $R$ between $C$ and $C'$ has non-empty intersection with $\alpha'$. Let $R_0$ be the component of $R \cap \calH^0$ incident to $\alpha'$. This either lies in $D'$ or has interior disjoint from $D'$. 

Suppose first that the interior of $R_0$ is disjoint from $D'$. (See the second case in Figure \ref{Fig:NewSimplifyingDisc}.) Then $R_0 \cup D'$ would have formed a parallelity, semi-parallelity or interpolating 0-handle of $S \cut C$. Moreover, in the case where $R_0 \cup D'$ is an interpolating handle, $R_0$ would have been incident to the parallelity bundle for $S \cut C$. Hence, when the interior of $R_0$ is disjoint from $D'$, $R_0 \cup D'$ would have been part of the isotopy region, contradicting the assumption that $D'$ is not part of the isotopy region.

Now suppose that $R_0$ lies in $D'$. (An example is shown in the first case in Figure \ref{Fig:NewSimplifyingDisc}.) In the case where $D'$ is an interior-simplifying disc, $R_0$ cuts $D'$ into at least two discs, and hence $D'$ contains at least two interior-simplifying discs for $C$, both of which are removed by the isotopy. In the case where $D'$ is a boundary-simplifying disc, $D' \cut R_0$ has at least one component disjoint from $\partial S$, and therefore $D'$ contains at least one interior-simplifying disc for $C$. Thus, in both cases, $f(D')$ is non-empty.
\end{proof}

%This sub-arc of $C$ that is involved in the maximal isotopy can be identified as follows.

%\begin{lemma}[{\sc Maximal isotopy}]
%There is an algorithm that takes as its input the following:
%\begin{enumerate}
%\item a handle structure $\calH$ for a compact orientable surface $S$;
%\item a standard 1-manifold properly embedded in $C$;
%\item an arc of $C \cap \calH^0$ transversely oriented into a simplifying disc $D$;
%\end{enumerate}
%and outputs the two points of $\partial(C \cut C')$ in $C$ where $C'$ is the 1-manifold obtained from $C$ by a maximal isotopy across $D$. It also provides the standard 1-manifold $C'$ that is obtained by the maximal isotopy, as well as the arc of $C' \cap \calH^0$ containing $C' \cut C$. The running time
%\end{lemma}

%\marginpar{Do we actually need this?}
%\begin{proof} As explained in Section \ref{}, we can compute the vector $(C)$. Using Theorem \ref{Thm:LocationParBundle}, we can also compute the vector for $\partial_h B$, where $B$ is the component of $\calB$ incident to $D$. The vector $(C) - (\partial_hB)$ nearly equals the vector for $C$. We must also remove $\partial D \cap C$. We can also compute the location of the component $\beta$ of $\partial_v B$ not incident to $D$. Let $\alpha_1$ and $\alpha_2$ be the components of $\calH^0 \cap C$ incident to $\beta$. Then we can form the arc $\alpha_1 \cup \beta \cup \alpha_2$. This might be parallel to other normal arcs of $C'$, in which case we include it in their pairing. If not, then we introduce a new pairing that just contains that arc. The resulting pairing system represents $C'$.
%\end{proof}

Let $D$ be a simplifying disc for $C$. Let $C'$ be the result of performing the isotopy across the isotopy region associated with $D$.
If ${\mathrm{simp}}(C') =  {\mathrm{simp}}(C)$, then one might hope to iterate this procedure, to remove the arc of $\calH^0 \cap C'$ that runs parallel to $D \cap C$, and so on. In fact, this is a good strategy, but the situation is complicated by the following possibility. The arc of $C' \cap \calH^0$ incident to the isotopy region might also be parallel to $D \cap C$. So, when we perform the isotopy across the isotopy region associated with $D$, we might not decrease the number of parallel copies of $D \cap C$. (See Figure \ref{Fig:SelfReturning}.) We formalise this as follows.

\begin{definition}
\label{Def:SelfReturning}
Let $\calH$ be a handle structure for a compact surface $S$, and let $C$ be a standard properly embedded 1-manifold.
Let $D$ be a simplifying disc for $C$ and let $R$ be the associated isotopy region. Let $C'$ be the 1-manifold obtained by the isotopy across $R$. Then $D$ and $R$ are \emph{self-returning} if the arc of $C' \cap \calH^0$ incident to $R$ is normally parallel to $D \cap C$ and all the components of $C \cap \calH^0$ between these two arcs are normally parallel to $D \cap C$.
\end{definition}

\begin{figure}[h]
\centering
\includegraphics[width=0.3\textwidth]{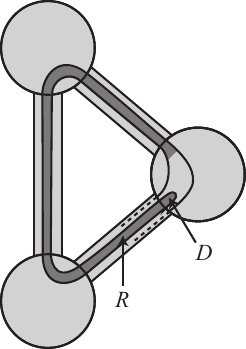}
\caption{The simplifying disc $D$ and its isotopy region $R$ are self-returning}
\label{Fig:SelfReturning}
\end{figure}

We will initially focus on the case where $D$ is not self-returning.

We know that if ${\mathrm{simp}}(C') =  {\mathrm{simp}}(C)$, then there is an arc of $\calH^0 \cap C'$ that runs parallel to $D \cap C$. If $D$ is not self-returning, then this is actually an arc of $\calH^0 \cap C$. It is natural to attempt to perform the associated isotopy at the same time as the one across $D$. We are therefore led to the following definition.

\begin{definition}
\label{Def:MaximalEnlargement}
Let $\calH$ be a handle structure for a compact surface $S$. Let $C$ be a standard 1-manifold properly embedded in $S$. Let $D$ be a simplifying disc for $C$. Let $R$ be the associated isotopy region. Then an \emph{enlargement} of $R$ consists of a disc $D_+$ containing $R$ and where $C \cap D_+$ consists of arcs all normally parallel to $R \cap C$, one of which is lies in $\partial D_+$. An enlargment is \emph{maximal} if it does lie in a bigger enlargement, up to isotopy preserving $C$ and the handles of $\calH$.
The \emph{isotopy across $D_+$} replaces the arcs $D_+ \cap C$ with arcs  parallel to $\partial R \cut (C \cup \partial S)$.
\end{definition}

\begin{figure}[h]
\centering
\includegraphics[width=0.6\textwidth]{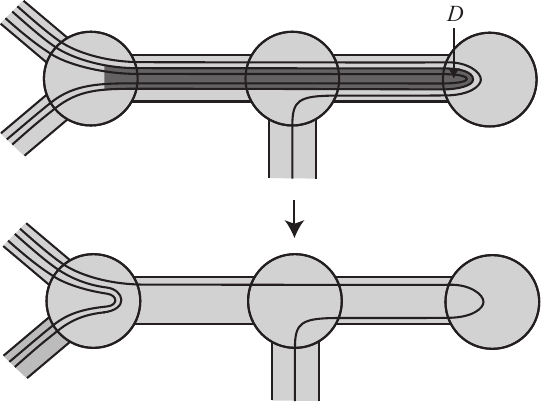}
\caption{The maximal enlargement of the isotopy region is shaded. The associated isotopy is shown.}
\label{Fig:MaximalEnlargement}
\end{figure}

Clearly, the isotopy across $D_+$ can be expressed as a composition of isotopies across isotopy regions. Hence, the following result is immediate.

\begin{lemma}
\label{Lem:MaxEnlargementIsotopySimp}
Let $\calH$ be a handle structure for a compact surface $S$, and let $C$ be a properly embedded standard 1-manifold. Let $C'$ be the 1-manifold that is obtained from $C$ by an isotopy across a maximal enlargement of an isotopy region. Then $\mathrm{simp}(C') \leq \mathrm{simp}(C)$.
\end{lemma} 

As is evident from Figure \ref{Fig:MaximalEnlargement}, the isotopy across $D_+$ might not reduce the simplification number. However, the following lemma asserts that one need only perform such isotopies a controlled number of times and then the simplification number must strictly decrease.

\begin{lemma}
\label{Lem:NotSelfReturning}
Let $\calH$ be a handle structure for a compact surface $S$. Let $C$ be a standard 1-manifold properly embedded in $S$. Let $D$ be a simplifying disc for $C$. Suppose that $D$ is not self-returning. Then after performing at most $18||\calH|| + 4\mathrm{simp}(C)$ isotopies across maximal enlargements of isotopy regions containing $D$, the resulting 1-manifold $C''$ satisfies $\mathrm{simp}(C'') < \mathrm{simp}(C)$.
\end{lemma}

\begin{proof}
Let $R$ be the isotopy region for $D$ and let $D_1$ be the maximal enlargement of $R$. Suppose first that every arc of $C \cap \calH^0$ parallel to $D \cap C$ is contained in $D_1$. Then, when we perform the isotopy across $D_1$, we remove all the normal arcs parallel to $D \cap C$. We are assuming that $D$ is not self-returning, and so this isotopy does not introduce any normal arcs parallel to $D \cap C$. Hence, this isotopy has reduced the simplification number, by Lemma \ref{Lem:SimpDoesNotGoUp}.

So, we may assume that there is a normal arc of $C \cap \calH^0$ parallel to $D \cap C$ that is not contained in $D_1$. Consider the outermost one in $\calH^0$, in other words, the one closest to $D$. Then we can extend this arc maximally, so that the resulting arc $\beta$ runs parallel to the boundary of $D_1$. But at some point this arc must diverge from $D_1$. At this point, $\partial D_1$ contains an arc of $C \cap \calH^0$ that is outermost in its parallelism class.
We enlarge the disc $D_1$ to form a disc $D_2$ that includes $\beta$ and all arcs normally parallel to it. Now consider the next arc of $C \cap \calH^0$ that is parallel to $D \cap C$ but not contained in $D_2$. It runs parallel to $\partial D_2$ until it diverges from it. Enlarge $D_2$ to form a disc $D_3$ that includes this arc and all arcs that are normally parallel to it. Continue in this way. The number of times that this can be iterated is at most $18||\calH|| + 4\mathrm{simp}(C)$. This because each step gives rise to an arc of $C \cap \calH^0$ that is outermost in its parallelism class. The number of such arcs is at most 
$18||\calH|| + 4\mathrm{simp}(C)$, by Lemma \ref{Lem:NumberOfPairingsCurve}. Hence after at most this number of isotopies, we end with a 1-manifold having smaller simplification number than $C$.
\end{proof}

The following algorithm establishes that we can implement the isotopy across the maximal enlargement of an isotopy region.

\begin{lemma}[{\sc Isotopy across maximal enlargement}]
\label{Lem:IsotopyAcrossMaximalEnlargement}
There is an algorithm that takes as its input the following:
\begin{enumerate}
\item a handle structure $\calH$ for a compact surface $S$;
\item a standard 1-manifold $C$ properly embedded in $S$;
\item an arc of $C \cap \calH^0$ transversely oriented into a simplifying disc $D$;
\end{enumerate}
and outputs the standard 1-manifold $C'$ that is obtained by performing the isotopy across the maximal enlargement of the isotopy region associated with $D$,
and possibly removing some boundary-parallel arc components. It also outputs the number of arcs of intersection between $C$ and this maximal enlargement.
The running time is bounded above by a polynomial function of $||\calH||$, $\mathrm{simp}(C)$, and $\log (w(C) + |\partial C|)$.
\end{lemma}

\begin{proof}
Let $R$ be the isotopy region associated with $D$ and let $\alpha$ be $\partial R \cap C$. By the algorithm {\sc Location of parallelity bundle component}
given in Theorem \ref{Thm:LocationParBundle}, we can determine the location on $C$ of the point or points $\partial \alpha - \partial S$. If $\partial \alpha - \partial S = \emptyset$, then $\alpha$ is a boundary-parallel arc component of $C$, in which case we simply remove $\alpha$ and all parallel components. So suppose that
$\partial \alpha - \partial S \not= \emptyset$.
We need to be able to determine the collection of arcs in $C$ that run parallel to $\alpha$. The point or points of $\partial \alpha - \partial S$ lie (after a small isotopy) in a component $\beta$ of $\calH^0 \cap \calH^1$. We plan to cut $\alpha$ along its intersection with a suitable sub-arc of $\beta$, so that the resulting components contain all the sub-arcs of $C$ parallel to $\alpha$. We might not be able to cut $C$ along the entirety of $C \cap \beta$ because in doing so, we might cut up $\alpha$ or one of its parallel arcs. We therefore need to determine the points of $\alpha \cap \beta$ that are closest to $\partial \alpha$. More precisely, we pick a point $x$ of $\partial \alpha - \partial S$ and we consider the component of $\beta'$ of $\beta \cut x$ that does not initially go into the isotopy region. We wish to determine the point of $(\beta' - x) \cap \alpha$ that is closest to $x$ or to determine that there is no such point. This can be achieved using AHT, as follows.

We start by cutting $C$ along all the points of $C \cap \beta'$. We need to determine whether this decomposes $\alpha$. We can do this by computing the components of $\alpha \cut \beta'$, identifying the components that contain $\partial \alpha$ and seeing whether there is just one such component (in which case $\alpha$ was not decomposed) or two (in which case $\alpha$ was decomposed). 
If we did not decompose $\alpha$, then we know that there are no other points of $\beta' \cap \alpha$ other than $x$. On other hand, suppose that we did decompose $\alpha$. In this case, we shrink $\beta'$ to a sub-arc containing $x$ that has half as many points of intersection with $C$ (rounded to a nearest integer). We cut $C$ along $\beta'$ and again determine using AHT whether this decomposes $\alpha$. If it does, then we shrink $\beta'$ again to half its size. On the other hand, if the decomposition along $\beta'$ did not decompose $\alpha$, then we enlarge $\beta'$ to $3/2$ of its former size. Repeating in this way, we can eventually home in on the point of $(\beta' - x) \cap \alpha$ closest to $x$.

In the case where there are no other points of $\beta \cap \alpha$ other than $\partial \alpha$, we simply cut $C$ along $C \cap \beta$, use AHT to determine its components and then count how many are parallel to $\alpha$. In the case where are points of $\beta \cap \alpha$ other than $\partial \alpha$, we let $d$ be the smallest number of points of $C \cap \beta$ strictly between $\partial \alpha$ and another point of intersection with $\beta \cap \alpha$. Let $\beta''$ be the sub-arc of $\beta$ containing $\partial \alpha$ and that extends $\lfloor d/2 \rfloor$ points of $C \cap \beta$ in each direction from $\partial \alpha$. We cut $C$ along $C \cap \beta''$ and then use AHT to determine the number of components parallel to $\alpha$.

In this way, we determine the number of parallel copies of $\alpha$ in the maximal enlargement of the isotopy region. We remove these from $C$ and replace them by the same number of short arcs in the 0-handle incident to $\beta$. The result is $C'$.
\end{proof}

We now consider the case where $D$ is self-returning. We would like to argue that the associated isotopy region $R$ can be combined with $D$ to form an embedded annulus or M\"obius band. However, there is a situation where this might not possible: when the horizontal boundary of $R$ contains an arc that is parallel to $D \cap C$. (See Figure \ref{Fig:SelfReturningTwice}.) We deal with this possibility in the following lemma.

\begin{figure}[h]
\centering
\includegraphics[width=0.6\textwidth]{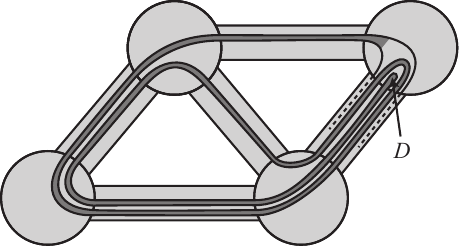}
\caption{The isotopy region $R$ for $D$ is shown shaded. It is self-returning. Its horizontal boundary contains arcs of intersection with $\calH^0$ that are parallel to $D \cap C$.}
\label{Fig:SelfReturningTwice}
\end{figure}

\begin{lemma}
\label{Lem:SelfReturningHoriz}
Let $\calH$ be a handle structure for a compact surface $S$. Let $C$ be a standard 1-manifold properly embedded in $S$. Let $D$ be a simplifying disc for $C$, and let $R$ be the associated isotopy region. Suppose that the horizontal boundary of $R$ contains a standard arc in $\calH^0$ that is parallel to $D \cap C$. Then after performing at most $36||\calH|| + 8 \mathrm{simp}(C)$ isotopies across maximal enlargements of isotopy regions containing $D$, the resulting 1-manifold $C''$ satisfies ${\mathrm{simp}}(C'') < {\mathrm{simp}}(C)$.
\end{lemma}

\begin{proof}
Let $\alpha$ be the sub-arc of $C$ that is the union of $D \cap C$ and the horizontal boundary of $R$.
By Lemma \ref{Lem:NotSelfReturning}, we may assume that $R$ is self-returning. As in the proof of Lemma \ref{Lem:NotSelfReturning}, we define a nested collection of discs embedded in $S$. The first is the maximal enlargement $D_1$ of $R$. If there is an arc of $C \cap \calH^0$ normally parallel to $D \cap C$ that is not contained in $D_1$, then we consider one that is outermost in $\calH^0$ and extend it as far as possible running parallel to $\partial D_1$, to form an arc $\beta$. Note that $\beta$ cannot run parallel to the entire length of $\alpha$, since in that case it would been part of $D_1$. We define the next disc $D_2$ to contain $D_1$ and to include $\beta$ and all normally parallel copies. We continue in this way until there are no more arcs of $C \cap \calH^0$ normally parallel to $D \cap C$. This forms a collection of discs $D_1 \subset D_2 \subset \dots \subset D_r$ where $r \leq 18||\calH|| + 4 \mathrm{simp}(C)$. We then perform all these isotopies, giving a 1-manifold $C'$. Since $R$ is self-returning, a component of $C' \cap \calH^0$ is normally parallel  to $C \cap D$. An outermost such arc in $\calH^0$ bounds a simplifying disc $D'$. This contains $D$. We claim that $D'$ is not self-returning. Assuming this claim, Lemma \ref{Lem:NotSelfReturning} then implies that after $18||\calH|| + 4 \mathrm{simp}(C)$ further isotopies across maximal enlargements of isotopy regions, we obtain a 1-manifold $C''$ satisfying $\mathrm{simp}(C'') < \mathrm{simp}(C') \leq \mathrm{simp}(C)$.

Suppose that the claim is not true. Consider the isotopy region $R'$ associated with $D'$. This consists of $D'$, followed by a component of the parallelity bundle or semi-parallelilty bundle of $S \cut C'$, possibly followed by an interpolation 0-handle, possibly followed by a component of the semi-parallelity bundle. The first 1-handle of $R'$ contains the first 1-handle of $R$. The next 0-handle of $R'$ contains the next 0-handle of $R$, and so on. Since we are assuming that $R'$ is self-returning, we deduce that there is a one-one correspondence between the handles of $R$ and the handles of $R'$, where the correspondence is that the former handle lies in the latter. In particular, each handle of $R'$ contains exactly one handle of $R$. But we are assuming that the horizontal boundary of $R$ contains a standard arc that is parallel to $D \cap C$. Hence, $R$ enters the space between $D \cap C$ and $D' \cap C'$. This implies that the first handle of $R'$ contains more than one handle of $R$. This is a contradiction.
\end{proof}

We are now left with the situation where $D$ is self-returning, but where the horizontal boundary of the isotopy region $R$ does not run parallel to $D \cap C$ at any point. The following lemma partially deals with this case.

\begin{lemma}
\label{Lem:SelfReturningNotReeb}
Let $\calH$ be a handle structure for a compact surface $S$. Let $C$ be a standard 1-manifold properly embedded in $S$. Let $D$ be a simplifying disc for $C$, and let $R$ be the associated isotopy region. Suppose $R$ is self-returning and that no arc of $\partial_hR$ is normally parallel to $D \cap C$. Let $D_+$ be the maximal enlargement of $R$. Suppose that $|D_+ \cap C|$ is strictly less than the number of arcs of $C \cap \calH_0$ normally parallel to $D \cap C$, including $D \cap C$. Let $C'$ be the 1-manifold obtained from $C$ by the isotopy across this maximal enlargement. Then $C' \cap \calH^0$ has a component that is parallel to $D \cap C$ and that bounds a simplifying disc for $C'$ which is not self-returning.
\end{lemma}

\begin{proof}
Since there are more arcs of $C \cap \calH_0$ normally parallel to $D \cap C$ than there are arcs of $C \cap D_+$, and since the horizontal boundary of $R$ does not run parallel to $D \cap C$ at any point, we deduce that some arc of $C \cap \calH^0$ parallel to $D \cap C$ is not moved in the isotopy. An outermost such arc bounds a simplifying disc $D'$ for $C'$. The associated isotopy region $R'$ cannot be self-returning, as otherwise $\partial_h R'$ would have been included in $D_+$.
\end{proof}

We still have not fully dealt with the case of self-returning isotopy regions. For such regions containing an interior-simplifying disc, the following structure will be useful.

\begin{definition}
\label{Def:Reeb}
Let $S$ be a compact surface with a handle structure $\calH$. Let $C$ be a standard 1-manifold properly embedded in $S$. A \emph{Reeb region} is a subsurface $A$ of $S$ with the following
properties:
\begin{enumerate}
\item $A$ is a thin regular neighbourhood of a standard simple closed curve, and hence is an annulus or M\"obius band;
\item the intersection between $A$ and $C$ consists of a union of essential arcs in $A$, all of which are normally parallel;
\item each such arc is a concatenation of three arcs $\alpha_-, \alpha_0, \alpha_+$;
\item $\alpha_-$ starts at a point of $\partial A \cap \calH^0 \cap \calH^1$, goes into $\calH^0$ and then runs normally parallel to $\partial A$ in some direction, possibly winding several times around that component but going at least once around that component, and then ending on the same component of $\calH^0 \cap \calH^1$ that is started on;
\item $\alpha_0$ is non-normal arc properly embedded in $\calH^0$ with endpoints in the same component of $\calH^0 \cap \calH^1$;
\item $\alpha_+$ starts at a point of $\partial A$ in the same component of $\calH^0 \cap \calH^1$ as $\partial \alpha_-$, then winds around $\partial A$ the same number of times as $\alpha_-$ and then ends on the same component of $\calH^0 \cap \calH^1$ that is started on.
\end{enumerate}
A \emph{Reeb isotopy} is supported in a Reeb region as above. It replaces $A \cap C$ with normally parallel non-normal arcs in $\calH^0$. A Reeb region $A$ is \emph{maximal} if, for any other Reeb region $A'$ containing it, there is an isotopy of $S$ preserving $C$ and each handle of $\calH$ taking $A'$ to $A$.
\end{definition}

\begin{figure}[h]
\centering
\includegraphics[width=0.8\textwidth]{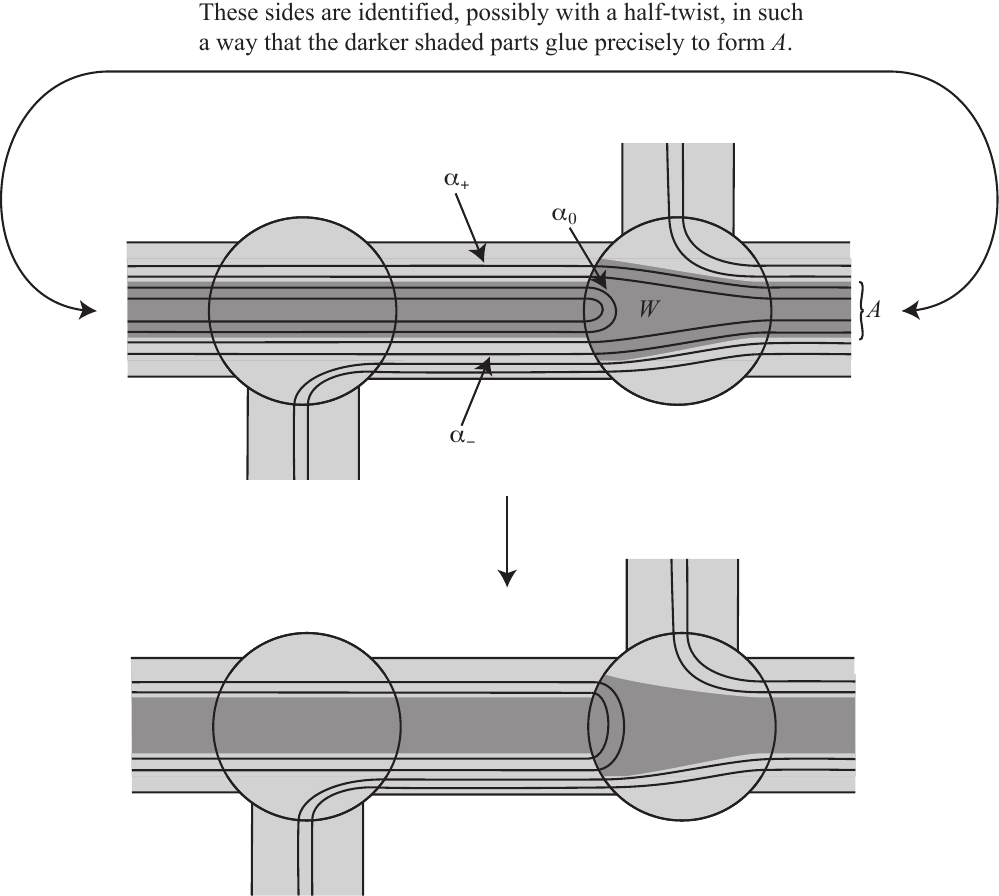}
\caption{A maximal Reeb region $A$ is shaded, and the associated Reeb isotopy is shown. The left and right sides of these figures are identified, but possibly via a reflection, in which case $A$ would be a M\"obius band}
\label{Fig:ReebRegion}
\end{figure}

In the case of a boundary-simplifying disc, we have the following structure.

\begin{definition}
\label{Def:SemiReeb}
Let $S$ be a compact surface with a handle structure $\calH$. Let $C$ be a standard 1-manifold properly embedded in $S$. A \emph{semi-Reeb region} is an annular subsurface $A$ of $S$ with the following
properties:
\begin{enumerate}
\item $A$ is a thin regular neighbourhood of a component of $\partial S$;
\item the intersection between $A$ and $C$ consists of a union of essential arcs in $A$, all of which are normally parallel;
\item each such arcs is a concatenation of two arcs $\alpha_0, \alpha_+$;
\item $\alpha_+$ starts at a point of $(\partial A - \partial S) \cap \calH^0 \cap \calH^1$, goes into $\calH^0$ and then runs normally parallel to $\partial A - \partial S$ in some direction, possibly winding several times around that component but going at least once around that component, and then ending on the same component of $\calH^0 \cap \calH^1$ that it started on;
\item $\alpha_0$ is non-normal arc properly embedded in $\calH^0$, ending on $\partial S$.
\end{enumerate}
A \emph{semi-Reeb isotopy} is supported in a semi-Reeb region as above. It replaces $A \cap C$ with normally parallel arcs in $\calH^0$. A semi-Reeb region $A$ is \emph{maximal} if, for any other semi-Reeb region $A'$ containing it, there is an isotopy of $S$ preserving $C$ and each handle of $\calH$ taking $A'$ to $A$.
\end{definition}

\begin{lemma}
Let $S$ be a compact surface with a handle structure $\calH$. Let $C$ be a standard 1-manifold properly embedded in $S$, and let $C'$ be the 1-manifold that results from a Reeb isotopy of $C$. Then $\mathrm{simp}(C') \leq \mathrm{simp}(C)$.
\end{lemma}

\begin{proof} 
We use the terminology of Definitions \ref{Def:Reeb} and \ref{Def:SemiReeb}. Suppose that each arc $\alpha_+$ winds $n$ times around $\partial A$. Then the Reeb or semi-Reeb isotopy is the composition of $n$ isotopies across maximal enlargements of isotopy regions. Hence, by Lemma \ref{Lem:MaxEnlargementIsotopySimp}, this does not increase the simplification number.
\end{proof}

Clearly, a Reeb or semi-Reeb isotopy across a maximal Reeb or semi-Reeb region $A$ does not necessarily reduce the simplification number since there is still a simplifying disc  in $A$ for the resulting 1-manifold. %However, this disc is not self-returning, and therefore Lemma \ref{Lem:NotSelfReturning} can then be applied to reduce the simplification number.

\begin{lemma}
\label{Lem:ReebConclusion}
Let $\calH$ be a handle structure for a compact surface $S$. Let $C$ be a standard 1-manifold properly embedded in $S$. Let $D$ be a simplifying disc for $C$, and let $R$ be the associated isotopy region. Suppose $R$ is self-returning and that no arc of $\partial_hR$ is normally parallel to $D \cap C$. Let $D_+$ be the maximal enlargement of $R$. Suppose that $|D_+ \cap C|$ is not strictly less than the number of arcs of $C \cap \calH_0$ normally parallel to $D \cap C$, including $D \cap C$. Then $D_+$ lies in a Reeb or semi-Reeb region for $C$.
\end{lemma}

\begin{proof}
Let $C'$ be the 1-manifold obtained by isotoping $C$ across $D_+$.
The arc $\partial D_+ \cap C$ is incident to a non-normal arc of $C' \cap \calH^0$. Let $E$ be the component of $\calH^0 \cut C$ that intersects $\calH^1$ in a single component and intersects $\partial S$ in at most one component. Then $D_+ \cup E$ is a Reeb annulus for $C$.
\end{proof}

\begin{lemma}
\label{Lem:ReebIsotopy}
Let $\calH$ be a handle structure for a compact surface $S$. Let $C$ be a standard 1-manifold properly embedded in $S$. Let $D$ be a simplifying disc for $C$, and suppose that $D$ lies in a maximal Reeb or semi-Reeb region $A$. Let $C'$ be the 1-manifold obtained by the Reeb or semi-Reeb isotopy, and let $D'$ be the resulting simplifying disc for $C'$. Suppose that $D'$ is self-returning. Then $D'$ satisfies the hypotheses of Lemma \ref{Lem:SelfReturningNotReeb}.
\end{lemma}

\begin{proof}
Let $R'$ be the isotopy region for $D'$. Since $R'$ is self-returning, it is a regular neighbourhood of the isotopy region for $D$. Hence, no arc of $\partial_h R' \cap \calH^0$ is normally parallel to $D' \cap C'$, which is one of the hypotheses of Lemma \ref{Lem:SelfReturningNotReeb}. Furthermore, suppose that the number of arcs of intersection between $C'$ and the maximal enlargement of $R'$ is not strictly less than the number of arcs of $C' \cap \calH^0$ normally parallel to $D' \cap C'$. Then we could have expanded $A$ to be a bigger Reeb or semi-Reeb region for $C$, contradicting its maximality. Thus, $D'$ does satisfy the hypotheses of Lemma \ref{Lem:SelfReturningNotReeb}.
\end{proof}

Thus, we are now in a position to give the algorithm {\sc Normalise a standard 1-manifold in a handle structure}.

\begin{proof}[Proof of Theorem \ref{Thm:Normalise1Manifold}]
We perform the following operations in a loop. First, we search for the existence of a maximal Reeb or semi-Reeb region. We will explain how to do this below. If there is such a Reeb or semi-Reeb region, we perform this Reeb or semi-Reeb isotopy, creating a 1-manifold $C'$. Denote the simplifying disc for $C'$ in the Reeb or semi-Reeb region by $D'$. 
If there was no Reeb or semi-Reeb region, then we let $C' = C$ and we pick any simplifying disc $D'$ for $C'$. 

We then perform the isotopy across the maximal enlargement of the isotopy region associated with $D'$. If this reduces the simplification number, then we return to the start of the loop. On the other hand, if the simplification number is unchanged, then we know that the new 1-manifold $C''$ has an arc of $C'' \cap \calH^0$ that runs parallel to $D' \cap C'$. An outermost such arc bounds a simplifying disc. We repeat the procedure in this paragraph with this disc. We continue until the simplification number decreases, at which point we return to the start of the loop. This continues until there are no more simplifying discs.

We now explain why this procedure works. If there is no Reeb or semi-Reeb region, then, by Lemma \ref{Lem:ReebConclusion}, each simplifying disc $D'$ for $C'$ satisfies at least one of the following: 
\begin{enumerate}
\item $D'$ is not self-returning;
\item the horizontal boundary of its isotopy region contains a standard arc that runs parallel to $D' \cap C'$;
\item the horizontal boundary of the isotopy region $R'$ does not contain a standard arc that runs parallel to $D' \cap C'$, and the number of arcs parallel to $D' \cap C'$ (including $D' \cap C'$) is greater than the number of arcs of $C'$ in the maximal enlargement of $R'$.
\end{enumerate}
In the first case, Lemma \ref{Lem:NotSelfReturning} gives that after performing at most $18||\calH|| + 4\mathrm{simp}(C)$ isotopies across maximal enlargements, the simplification number goes down. In the second case, a similar conclusion holds by Lemma \ref{Lem:SelfReturningHoriz}. In the third case, Lemma \ref{Lem:SelfReturningNotReeb} implies that the 1-manifold obtained by isotoping $C'$ across the maximal enlargement of $R'$ has a simplifying disc that is not self-returning. Hence, after at most $18||\calH|| + 4\mathrm{simp}(C)$ further isotopies, the simplification number goes does.
On other hand, if there is a Reeb or semi-Reeb region for $C$, then there is a maximal one $A$, and we let $C'$ be the 1-manifold obtained by the Reeb or semi-Reeb isotopy and we let $D'$ be the simplifying disc for $C'$ contained in $A$. Lemma \ref{Lem:ReebIsotopy} and Lemma \ref{Lem:SelfReturningNotReeb} imply that either $D'$ is not self-returning or the 1-manifold obtained by the isotopy across the maximal enlargement of $D'$ has such a simplifying disc. Hence, after at most $18||\calH|| + 4\mathrm{simp}(C)$ further isotopies, the simplification number decreases.

We now describe how to find a maximal Reeb region if one exists. We can readily find all the interior-simplifying discs for the 1-manifold $C$, by considering each parallelism class of arcs of $C \cap \calH^0$. We consider each interior-simplifying disc in turn. So let us focus on one such disc $D$ and its isotopy region $R$. We first determine whether $R$ is self-returning as follows. By applying {\sc Cut along standard 1-manifold} (Theorem \ref{Thm:CutAlong1Manifold}) to $S \cut C$ and Remark \ref{Rem:ReducedHSEmbeds}, we can  determine the component $W$ of $\calH^0 \cut C$ containing $\partial_v R \cut D$. (See Figure \ref{Fig:ReebRegion}.) Then $R$ is self-returning if and only if the following all hold:
\begin{enumerate}
\item $W$ lies in the interior of $S$;
\item the intersection between $\partial W$ and $\calH^0 \cap C$ consists of three arcs $\beta_-$, $\beta_0$ and $\beta_+$;
\item one of these arcs $\beta_0$ is normally parallel to $D \cap C$ and all arcs of $C \cap \calH^0$ between $\beta_0$ and $D \cap C$ are normally parallel;
\item the remaining two arcs $\beta_-$ and $\beta_+$ of $\calH^0 \cap C \cap \partial W$ are normally parallel to each other.
\end{enumerate}
These conditions are easily checked. Suppose that $R$ is self-returning. Then we can easily check whether the horizontal boundary of $R$ contains no arc of $\calH^0 \cap C$ that is normally parallel to $D \cap C$ using {\sc Location of parallelity bundle component} (Theorem \ref{Thm:LocationParBundle}). Using {\sc Isotopy across maximal enlargement} (Lemma \ref{Lem:IsotopyAcrossMaximalEnlargement}), we can determine the number $p$ of arcs of intersection between $C$ and the maximal enlargement of $R$. For $R$ to be part of a Reeb region, this number needs to be equal to the number of components of $C \cap \calH^0$ normally parallel to $C \cap D$, including $C \cap D$ itself. 

Thus, using the above procedure, we can determine whether $R$ is part of a Reeb region. Suppose it is. We then find the maximal Reeb region containing $R$, as follows. Let $\beta$ be the component of $\calH^0 \cap \calH^1$ incident to $D$. As in the proof of Lemma \ref{Lem:IsotopyAcrossMaximalEnlargement}, we can determine whether $R$ intersects $\beta - D$. If it does not, then we cut $C$ along $\beta$ and count the number of components of $C - \beta$ normally parallel to each component of $\partial_h R$. Let $q_1$ and $q_2$ be these number of arcs. Let $q$ be the largest integer multiple of $p$ that is at most $\min \{ q_1, q_2 \}$. Then the union of this many copies of $\partial_h R$, together with all arcs of $C \cap \calH^0$ parallel to $C \cap D$, forms the arcs of intersection between $C$ and the maximal Reeb region $A$. Removing all these arcs and replacing them with this many copies of $C \cap D$ achieves the Reeb isotopy. 

There is a similar procedure for finding a maximal semi-Reeb region if one exists and for performing the associated semi-Reeb isotopy.
\end{proof}

\begin{remark}
\label{Rem:CompositionSmallIsotopies}
The isotopies given in the above proof can be realised as a union of isotopies across simplifying discs and their incident 1-handles. However, of course, the number of such isotopies may be considerably higher than the number of moves in the above algorithm.
\end{remark}

\section{Minimal position for curves in surfaces}
\label{Sec:MinimalPositionCurves}

In this section, we embark on the computation of the geometric intersection number of two properly embedded 1-manifolds $P$ and $C$. In fact, we will actually isotope $C$ into minimal position with respect to $P$, but in the special case where $P$ has very restricted intersection with the triangulation or handle structure.

\begin{definition} 
Let $C$ and $P$ be 1-manifolds properly embedded in a compact surface $S$ that intersect transversely. A \emph{bigon} for $C$ and $P$ is a disc $D$ embedded in the interior of $S$ with interior disjoint from $C \cup P$, and with boundary that is the concatenation of an arc in $C$ and an arc in $P$. A \emph{half-bigon} a disc $D$ embedded in $S$ with interior disjoint from $C \cup P$, and with boundary that is the concatenation of an arc in $C$, an arc in $P$ and an arc in $\partial S$. The \emph{bigon number} $\mathrm{bigon}(C,P)$ is equal to the sum of the number of half-bigons and twice the number of bigons.
We say that $C$ and $P$ are in \emph{minimal position} if they have no bigon or half-bigon.
\end{definition}

The following key fact is well known \cite[Proposition 1.7 and Section 1.2.7]{FarbMargalit}.

\begin{lemma}
\label{Lem:MinimalNoBigon}
Let $C$ and $P$ be 1-manifolds properly embedded in a compact surface $S$ that intersect transversely. Then $|C \cap P|$ is minimised up to isotopy if and only if they have no bigon or half-bigon.
\end{lemma}

\begin{theorem}[{\sc Minimal Position for 1-Manifolds}]
\label{Thm:MinimalPosition}
There is an algorithm that takes, as its input, the following:
\begin{enumerate}
\item a triangulation $\calT$ or handle structure $\calH$ of a compact surface $S$;
\item a 1-manifold $P$ properly embedded in $S$ that either is a subset of the 1-skeleton of $\calT$ or that intersects $\calH$ in the following way: it is disjoint from 0-handles; it intersects each 1-handle in the empty set or in the co-core of the handle; it intersects each 2-handle in at most one properly embedded arc;
\item a standard 1-manifold $C$;
\end{enumerate}
and provides as its output a normal 1-manifold $C'$ that is obtained from $C$ by an isotopy, plus possibly removing some components that are boundary parallel arcs and curves that bound discs. This 1-manifold $C'$ is in minimal position with respect to $P$ and has weight bounded above by $w(C)(1 + 2|\calT|)$ or $w(C)(1 + 2||\calH||)$. The running time is bounded above by a polynomial function of $\mathrm{simp}(C)$, $\log (w(C) + |\partial C|)$, and $|\calT|$ or $||\calH||$.
\end{theorem}

When $P$ is arranged with respect to the handle structure $\calH$ as in (2) above, we say that it is \emph{cellular with respect to the cell structure dual to $\calH$}.

The proof will follow that in Section \ref{Sec:Normalise}, but with bigons and half-bigons replacing the role of simplifying discs. In Section \ref{Sec:Normalise}, the primary measure of complexity of the standard 1-manifold $C$ was the simplification number $\mathrm{simp}(C)$, but here we use $\mathrm{bigon}(C,P)$. The reader may wonder why the running time for {\sc Minimal Position for 1-Manifolds} does not depend on $\text{bigon}(C,P)$, given that this quantity is our primary measure of complexity. In fact, the running time does depend on $\text{bigon}(C,P)$, but this can be bounded above in terms of $||\calH||$ and $\text{simp}(C)$, as follows.

\begin{lemma}
\label{Lem:BoundBigon}
Let $S$ be a compact surface with a handle structure $\calH$. Let $P$ be a properly embedded 1-manifold that is cellular with respect to the cell structure dual to $\calH$, and let $C$ be a standard 1-manifold. Then
$$\mathrm{bigon}(C,P) \leq 20 ||\calH|| + 4 \mathrm{simp}(C).$$
\end{lemma}

\begin{proof}
We obtain a handle structure $\calH'$ for $S \cut P$ by removing the handles intersecting $P$. Hence, $||\calH'|| \leq ||\calH||$. The 1-manifold $C \cut P$ is standard in $\calH'$. Each simplifying disc for $C \cut P$ is a simplifying disc for $C$, but it may have become boundary-simplifying when previously it was interior-simplifying. Hence, $\mathrm{simp}(C \cut P) \leq \mathrm{simp}(C)$. The components of $C \cut P$ fall into at most $10 ||\calH'|| + 2 \mathrm{simp}(C \cut P)$ parallelism classes, by Lemma \ref{Lem:NumberOfCurveComponents}. In each parallelism class, only the outermost two 1-manifolds can be part of a bigon or half-bigon. The lemma follows immediately.
\end{proof}

%It is worth recording the following consequence.

\begin{corollary} 
Let $S$ be a compact surface with a triangulation $\calT$. Let $P$ be a properly embedded 1-manifold that is a subset of the 1-skeleton, and let $C$ be a normal 1-manifold. Then $\mathrm{bigon}(C,P) \leq  154 |\calT|$.
\end{corollary}

\begin{proof} Dual to $\calT$ is a handle structure $\calH$. Each handle of $\calH$ arises from a simplex of $\calT$ that does not lie wholly in $\partial S$. The number of edges of $\calT$ not in $\partial S$ is at most $(3/2)|\calT|$. The number of vertices of $\calT$ is at most $3|\calT|$. Hence, the number of handles of $\calH$ is at most $(11/2)|\calT|$. The 1-manifold $C$ might not be normal with respect to $\calH$, but it has only boundary-simplifying discs and at most two of these can lie in each 0-handle. So by Lemma \ref{Lem:BoundBigon}, $\mathrm{bigon}(C,P) \leq 20 ||\calH|| + 4 \mathrm{simp}(C) \leq 28 ||\calH|| \leq 154 |\calT|$.
\end{proof}

In {\sc Minimal Position for 1-Manifolds}, our aim will be to isotope $C$ to a 1-manifold $C'$ with $\mathrm{bigon}(C',P) = 0$. Just as in Figure \ref{Fig:IncreaseSimplifyingIndex}, an isotopy of $C$ across a bigon or half-bigon may result in a 1-manifold $C'$ with $\mathrm{bigon}(C',P) > \mathrm{bigon}(C,P)$. In order to avoid this, we need to consider a more substantial move. We therefore need to introduce an analogue of the isotopy region for a simplifying disc, and to do so, we will introduce analogues of many of the definitions in Sections \ref{Sec:HSNormal}, \ref{Sec:Cut} and \ref{Sec:Normalise}. We will use the adjective `topological' for many of these analogues.

\begin{definition}
A \emph{topological parallelity region} for $C$ with respect to $P$ consists of a disc with interior disjoint from $C$ and $P$, and with boundary equal to the concatenation of an arc in $C$, an arc in $P$ or $\partial S$, an arc in $C$ and another arc in $P$ or $\partial S$. The \emph{topological parallelity bundle} is the subset of $S \cut C$ consisting of the union of the topological parallelity regions. It is an $I$-bundle $\calB$ over a 1-manifold, with $\partial_h \calB = \calB \cap C$. See Figure \ref{Fig:TopParRegion}.
\end{definition}

\begin{definition}
A \emph{topological semi-parallelity region} for $C$ with respect to $P$ is a disc with interior disjoint from $C$ and $P$, and with boundary equal to the concatenation of an arc in $C$, an arc in $P$, an arc in $\partial S$ and another arc in $P$. The \emph{topological semi-parallelity bundle} is the subset of $S \cut C$ consisting of the union of the semi-parallelity regions. It is an $I$-bundle $\calB$ over a 1-manifold, with $\partial_h \calB = \calB \cap (C \cup \partial S)$.
\end{definition}

\begin{definition}
A \emph{topological interpolating region} for $C$ and $P$ is a disc with interior disjoint from $C$ and $P$, and with boundary consisting of an arc in $C$, an arc in $P$, an arc in $C$, an arc in $P$ and an arc in $\partial S$.
\end{definition} 

\begin{figure}[h]
\centering
\includegraphics[width=0.75\textwidth]{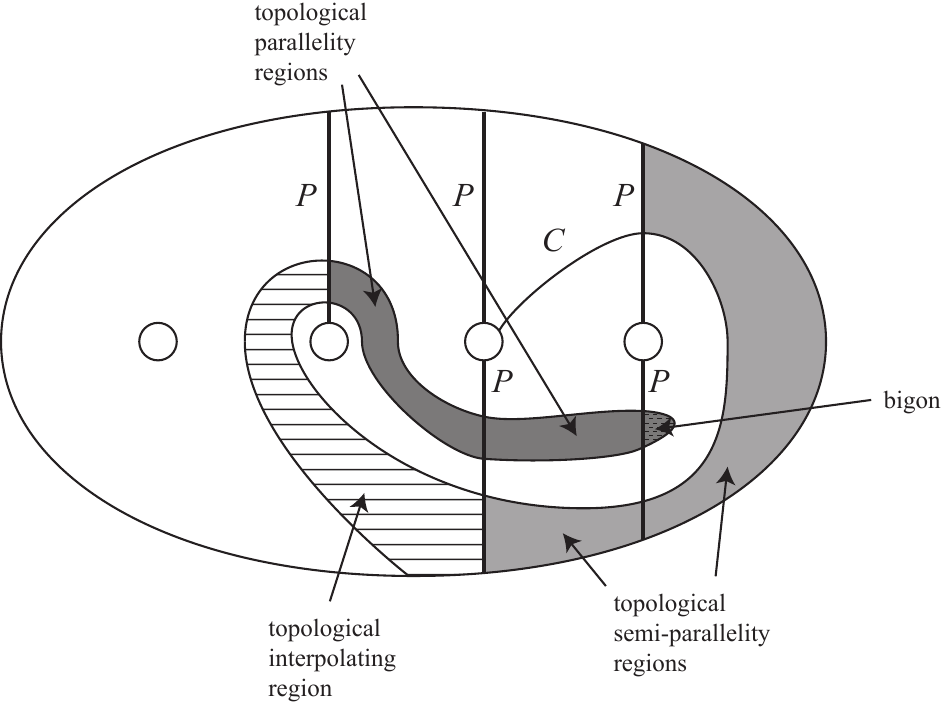}
\caption{A surface $S$ containing properly embedded 1-manifolds $C$ and $P$. Shown are various complementary regions, which are topological parallelity regions, topological semi-parallellity regions, a topological interpolating region and a bigon.}
\label{Fig:TopParRegion}
\end{figure}

\begin{definition}
Let $D$ be a bigon or half-bigon for $C$ and $P$. Then the \emph{associated topological isotopy region} $R$ is the union of the following subsets of $S \cut C$:
\begin{enumerate}
\item the bigon or half-bigon $D$;
\item the component $B_1$ of the topological parallelity or topological semi-parallelity bundle incident to $D$;
\item if $D$ is a bigon and $B_1$ is incident to a topological interpolating region, then this topological interpolating region is also part of $R$ and so is the incident component $B_2$ of the topological semi-parallelity bundle;
\item a small regular neighbourhood of $\partial_v B_1$ in $S \cut C$, as well as a small regular neighbourhood of $\partial_v B_2$ if $B_2$ is defined.
\end{enumerate}
We say that the \emph{horizontal boundary} of $R$ is $(\partial R \cap C) \cut D$.
\end{definition}

\begin{definition}
We say that two arcs of $C \cut P$ in $S \cut P$ are \emph{topologically parallel} if they cobound a topological parallelity region. More generally, if $\alpha_1$ and $\alpha_2$ are sub-arcs of $C$ with $\partial \alpha_1 \subset P$ and $\partial \alpha_2 \subset P$, we say that these arcs are \emph{topologically parallel} if there is a disc $R$ in $S$, such that $\partial R$ is the concatenation of $\alpha_1$, an arc in $P$ or $\partial S$, $\alpha_2$ and another arc in $P$ or $\partial S$, and $R$ is a union of topological parallelity regions.
\end{definition}

The following analogue of Lemma \ref{Lem:SimpDoesNotGoUp} holds, with essentially the same proof, which is omitted.

\begin{lemma}
\label{Lem:BigonDoesNotGoUp}
Let $C$ and $P$ be 1-manifolds properly embedded in a compact surface $S$ that intersect transversely.
Let $D$ be a bigon or half-bigon for $C$ and $P$.
Let $C'$ be the 1-manifold obtained by isotoping $C$ across the topological isotopy region associated with $D$. Then $\mathrm{bigon}(C', P) \leq  \mathrm{bigon}(C,P)$ and $|C' \cap P| < |C \cap P|$. Furthermore, if $\mathrm{bigon}(C',P) =  \mathrm{bigon}(C,P)$, then there is an arc of $C' \cut P$ that is topologically parallel to $D \cap C$, and all the components of $C \cut P$ between these two arcs are topologically parallel to $D \cap C$.
\end{lemma}

The following is the version of Definition \ref{Def:MaximalEnlargement}.

\begin{definition}
\label{Def:MaximalEnlargementBigon}
Let $C$ and $P$ be 1-manifolds properly embedded in $S$ that intersect transversely. Let $D$ be a bigon or half-bigon for $C$ and $P$. Let $R$ be the associated topological isotopy region. Then an \emph{enlargement} of $R$ consists of a disc $D_+$ containing $R$ and where $C \cap D_+$ consists of arcs all parallel to $R \cap C$, one of which is lies in $\partial D_+$. An enlargment is \emph{maximal} if it does lie in a bigger enlargement, up to an isotopy that preserves $C$ and $P$.
The \emph{isotopy across $D_+$} replaces the arcs $D_+ \cap C$ with arcs  parallel to $\partial R \cut (C \cup \partial S)$.
\end{definition}

We also have the analogue of {\sc Isotopy across maximal enlargement} (Lemma \ref{Lem:IsotopyAcrossMaximalEnlargement}) and {\sc Location of parallelity bundle component} (Theorem \ref{Thm:LocationParBundle}).

\begin{lemma}[{\sc Isotopy across maximal enlargement of bigon/half-bigon}]
There is an algorithm that takes as its input the following:
\begin{enumerate}
\item a triangulation $\calT$ or handle structure $\calH$ for a compact surface $S$;
\item a properly embedded essential 1-manifold $P$ that either is a subset of the 1-skeleton of $\calT$ or is cellular with respect to the cell complex dual to $\calH$;
\item a standard 1-manifold $C$ properly embedded in $S$;
\item an arc of $C \cap \calH^0$ transversely oriented into a bigon or half-bigon $D$ for $C$ and $P$;
\end{enumerate}
and outputs the standard 1-manifold $C'$ that is obtained by the isotopy across the maximal enlargment of the isotopy region associated with $D$ and possibly removing some boundary-parallel arc components. It also provides the location of the sub-arc of $P$ that the components of $C' \cut C$ run along. The running time is bounded above by a polynomial function of $|\calT|$ or $||\calH||$, $\mathrm{simp}(C)$, and $\log (w(C) + |\partial C|)$.
\end{lemma}

We also have the analogue of Definition \ref{Def:SelfReturning}.

\begin{definition}
\label{Def:SelfReturningBigon}
Let $D$ be a bigon or half-bigon for $C$ and $P$, and let $R$ be the associated topological isotopy region. Let $C'$ be the 1-manifold obtained by the isotopy across $R$. Then $D$ and $R$ are \emph{self-returning} if the arc of $C' \cut P$ incident to $R$ is topologically parallel to $D \cap C$, and all the components of $C \cut P$ between these two arcs are topologically parallel to $D \cap C$.
\end{definition}

The following is a version of Lemma \ref{Lem:NotSelfReturning}, with almost the same proof.

\begin{lemma}
\label{Lem:NotSelfReturningBigon}
Let $C$ and $P$ be 1-manifolds properly embedded in a compact surface $S$ that intersect transversely. Let $D$ be a bigon or half-bigon for $C$ and $P$. Suppose that $D$ is not self-returning. Then after performing at most $-8 \chi(S) + 3 \mathrm{bigon}(C,P) + |\partial P|$ isotopies across maximal enlargements of topological isotopy regions containing $D$, the resulting 1-manifold $C''$ satisfies $\mathrm{bigon}(C'',P) < \mathrm{bigon}(C,P)$.
\end{lemma}

The only deviation from the proof of Lemma \ref{Lem:NotSelfReturning} is that instead of using Lemma \ref{Lem:NumberOfPairingsCurve}, we use the following.

\begin{lemma}
\label{Lem:NumberNonParallelRegions}
Let $C$ and $P$ be 1-manifolds properly embedded in the compact surface $S$ that intersect transversely. Suppose that no component of $C$ or $P$ bounds a disc in the interior of $S$ with interior disjoint from $C \cup P$, and no component is topologically parallel to an arc in $\partial S$, via a disc with interior disjoint from $C \cup P$. 
%Suppose also that no component of $S \cut (C \cup P)$ is an annulus or M\"obius band disjoint from $\partial S$, and that no component is a disc with boundary consisting of an arc in $C$, an arc in $\partial S$, an arc in $P$ and an arc in $\partial S$. 
%Then the number of components of $S \cut (C \cup P)$ that are not parallelity or semi-parallelity regions is at most $-4 \chi(S) + 2 \mathrm{bigon}(C,P)$. 
Then the number of topological parallelism classes of components of $C \cut P$ that are not disjoint from $P$ is at most $-4\chi(S) + (3/2) \mathrm{bigon}(C,P) + |\partial P|/2$.
\end{lemma}

\begin{proof}
First note that we may remove components of $C$ that are disjoint from $P$. Similarly, we may remove components of $P$ that are disjoint from $C$.
Define the \emph{index} of a component $R$ of $S \cut (C \cup P)$ to be 
$$I(R) = -\chi(R) + |R \cap C \cap P|/4 + |R \cap C \cap \partial S|/4 + |R \cap P \cap \partial S|/4.$$ 
Then it is an easy calculation that the sum, over all components $R$, of the index of $R$, is equal to $-\chi(S)$. We now consider the regions $R$ with negative index. Necessarily, such a region $R$ must be a disc intersecting $(C \cap P) \cup (C \cap \partial S) \cup (P \cap \partial S)$ at most three times. By our hypotheses, it is therefore a  bigon or half-bigon. A bigon has index $-1/2$ and a half-bigon has index $-1/4$. Hence,
$$\sum_{I(R) > 0} I(R) = - \chi(S) + \mathrm{bigon}(C,P)/4.$$
But each region $R$ that is not a topological parallelity or topological semi-parallelity region and is not a bigon or half-bigon has positive index at least $1/4$. (Here we are using the assumption that every component of $C$ intersects $P$ and that every component of $P$ intersects $C$. This is to rule out disc regions $R$ consisting of an arc in $\partial S$, an arc of $C$ or $P$, another arc in $\partial S$ and another arc of $C$ or $P$. It also rules out annuli and M\"obius band components with boundary disjoint from $C \cap P$.) Therefore, the number of such regions is at most $-4\chi(S) + \mathrm{bigon}(C,P)$. The number of regions with negative index is at most $\mathrm{bigon}(C,P)$. Hence, in total, the number of regions that are not parallelity or semi-parallelity regions is at most $-4 \chi(S) + 2 \mathrm{bigon}(C,P)$.

To bound the number of topological parallelism classes of components of $C \cut P$, we replace each topological parallelism class with a single component. We view the resulting 1-manifold $\alpha$ as properly embedded in $S \cut P$. The complementary regions of $\alpha$ correspond to the complementary regions of $C \cup P$ that are not topological parallelity regions. Each such region $R$ with positive index has at most $8 I(R)$ arcs of $C$ in its boundary, where the extreme case is given by a topological interpolating region with index $1/4$ and with 2 arcs of $C$ in its boundary. So, the number of arcs of $\alpha$ in the boundary of the regions with positive index is at most $8 \sum_{I(R) > 0} I(R)$, which is at most $- 8\chi(S) + 2 \mathrm{bigon}(C,P)$. The number of arcs of $\alpha$ in the boundary of the regions with negative index is at most $\mathrm{bigon}(C,P)$. The number of arcs of $\alpha$ in the boundary of regions with zero index is at most the number of topological semi-parallelity regions, which is at most $|\partial P|$. So in total, the number of arcs of $\alpha$ in the boundary of these regions is at most $- 8\chi(S) + 3 \mathrm{bigon}(C,P) + |\partial P|$. Each arc of $\alpha$ is counted twice, and hence we find the required upper bound on the number of topological parallelism classes of $C \cut P$.
\end{proof}

The following is an analogue of Lemma \ref{Lem:SelfReturningHoriz}.

\begin{lemma}
\label{Lem:SelfReturningHorizBigon}
Let $C$ and $P$ be 1-manifolds properly embedded in a compact surface $S$ that intersect transversely. Let $D$ be a bigon or half-bigon for $C$ and $P$. Let $R$ be the associated topological isotopy region. Suppose that $\partial R$ contains an arc of $C \cut P$ that is topological parallel to $D \cap C$ but not equal to $D \cap C$. Then after performing at most $-16\chi(S) + 6 \mathrm{bigon}(C,P) + 2 |\partial P|$ isotopies across maximal enlargements of topological isotopy regions containing $D$, the resulting 1-manifold $C''$ satisfies ${\mathrm{bigon}}(C'', P) < {\mathrm{bigon}}(C,P)$.
\end{lemma}

There is also an obvious analogue of Lemma \ref{Lem:SelfReturningNotReeb}, which we omit.

Finally, we introduce the analogue of a Reeb region in this context.

\begin{definition}
Let $S$ be a compact surface containing properly embedded 1-manifolds $P$ and $C$ that intersect transversely. A \emph{topological Reeb region} for $C$ with respect to $P$ is a subsurface $A$ in the interior of $S$ with the following properties:
\begin{enumerate}
\item $A$ is an annulus or M\"obius band;
\item $A \cap P$ is a collection of arcs embedded in $A$, where each such arc is a concatenation of an arc or point in $\partial A$, a properly embedded essential arc in $A$, and another arc or point in $\partial A$;
\item $A$ contains a single bigon $D$ in its interior;
\item the intersection between $A$ and $C$ consists of a union of essential arcs in $A$, all of which are topologically parallel;
\item each such arc is a concatenation of three arcs $\alpha_-, \alpha_0, \alpha_+$;
\item $\alpha_-$ starts at a point of $\partial A \cap P$ in the same component of $A \cap P$ as $D \cap P$ and then runs topologically parallel to $\partial A$ in some direction, possibly winding several times around that component, but running at least once along that component;
\item $\alpha_0$ is topologically parallel to (or equal to) $D \cap C$;
\item $\alpha_+$ starts at a point of $\partial A \cap P$ in the same component of $A \cap P$ as $D \cap P$, and winds around $\partial A$ the same number of times as $\alpha_-$ and in the same direction.
\end{enumerate}
A \emph{topological Reeb isotopy} is supported in a topological Reeb region as above. It replaces $A \cap C$ with parallel arcs that are essential in $A$ and with interiors disjoint from $P$. A topological Reeb region $A$ is \emph{maximal} if any other topological Reeb region containing it is isotopic to $A$, via an isotopy preserving $P$ and $C$.
\end{definition}

We also have the following analogue of a semi-Reeb region.

\begin{definition}
Let $S$ be a compact surface, and let $C$ and $P$ be 1-manifolds properly embedded in $S$ that intersect transversely. A \emph{topological semi-Reeb region} for $C$ with respect to $P$ is an annular subsurface $A$ of $S$ with the following
properties:
\begin{enumerate}
\item $A$ is a regular neighbourhood of a component of $\partial S$;
\item $A \cap P$ is a collection of arcs embedded in $A$, where each such arc is a concatenation of an arc or point in $\partial A - \partial S$ and a properly embedded essential arc in $A$;
\item $A$ contains a single half-bigon $D$ and no bigon;
\item the intersection between $A$ and $C$ consists of a union of essential arcs in $A$, all of which are topologically parallel;
\item each such arcs is a concatenation of two arcs $\alpha_0, \alpha_+$;
\item $\alpha_+$ starts at a point of $(\partial A - \partial S) \cap P$ in the same component of $A \cap P$ as $D \cap P$ and then runs topologically parallel to $\partial A - \partial S$ in some direction, possibly winding several times around that component, but running at least once along that component;
\item $\alpha_0$ is topologically parallel to (or equal to) $\partial D \cut (\partial S \cup P)$.
\end{enumerate}
(See Figure \ref{Fig:TopSemiReeb}.)
A \emph{semi-Reeb isotopy} is supported in a semi-Reeb region as above. It replaces $A \cap C$ with parallel arcs that are essential in $A$ and with interiors disjoint from $P$.
A semi-Reeb region $A$ is \emph{maximal} if any other semi-Reeb region containing it is isotopic to $A$, via an isotopy preserving $P$ and $C$.
\end{definition}

\begin{figure}[h]
\centering
\includegraphics[width=0.75\textwidth]{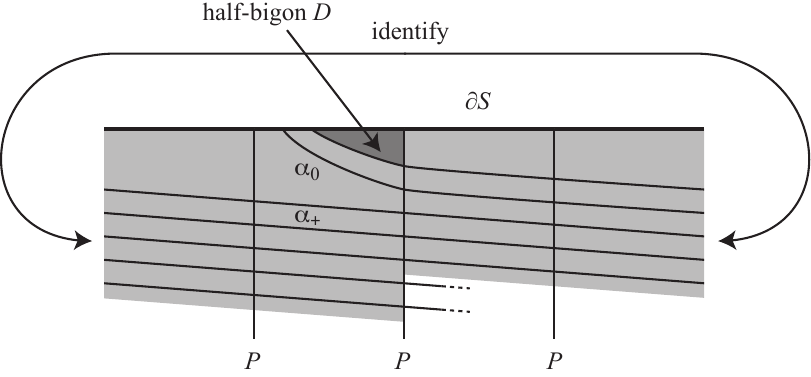}
\caption{A topological semi-Reeb region}
\label{Fig:TopSemiReeb}
\end{figure}

There are obvious analogues of Lemmas \ref{Lem:ReebConclusion} and \ref{Lem:ReebIsotopy} for topological Reeb or semi-Reeb regions.

The following lemma controls the increase in the simplification number of the 1-manifold $C$ as it is modified by the above isotopies.

\begin{lemma}
\label{Lem:IncreaseInSimp}
Let $S$ be a compact surface with a handle structure $\calH$. Let $P$ be a 1-manifold properly embedded in $S$ that is cellular with respect to the cell structure dual to $\calH$.
Let $C$ be a standard 1-manifold properly embedded in $S$. Let $C'$ be obtained from $C$ either by an isotopy across the maximal enlargement of an isotopy region associated with a bigon or half-bigon or by a Reeb or semi-Reeb isotopy associated with a bigon or half-bigon. Then $\mathrm{simp}(C') \leq \mathrm{simp}(C) + 8 ||\calH||$ and $w(C') \leq w(C) + 2 ||\calH||(|C \cap P| - |C' \cap P|)$.
\end{lemma}

\begin{proof}
In the isotopy, sub-arcs of $C$ are replaced by arcs normally parallel to a sub-arc of an arc of $P \cut C$. Specifically, this arc of $P \cut C$ is shifted out of the 2-handles and 1-handles so that it becomes an arc $\alpha$ running through the 0-handles and 1-handles. When this arc is shifted into the 0-handles and 1-handles, it runs over each 1-handle at most twice. So the number of components of $\alpha \cap \calH^0$ is at most $2|\calH^1| \leq 2||\calH||$. 
Each component of $\alpha \cap \calH^0$ may give rise to multiple components of $(C'\cut C) \cap \calH^0$, but only the outermost two of these can be part of a simplifying disc for $C'$. Thus the number of new simplifying discs that are created is at most $4 ||\calH||$. Each simplifying disc contributes at most $2$ to $\mathrm{simp}(C')$ and so $\mathrm{simp}(C') \leq \mathrm{simp}(C) + 8 ||\calH||$.

The number of arcs of $C' \cut P$ that are inserted is at most $|P \cap C| - |P \cap C'|$, because each new arc that is inserted leads to a reduction in the number of intersections with $P$ by at least 1. As observed above, each of these new arcs has weight at most $2||\calH||$. Hence, $w(C') \leq w(C) + 2 ||\calH||(|C \cap P| - |C' \cap P|)$.
\end{proof}

\begin{proof}[Proof of Theorem \ref{Thm:MinimalPosition}] 
First we search for a maximal Reeb or semi-Reeb region. If there is one, then we perform the Reeb or semi-Reeb isotopy, creating a 1-manifold $C'$. Let $D'$ be the bigon or half-bigon for $C'$ in the Reeb or semi-Reeb region. If there was no Reeb or semi-Reeb region, then we pick any bigon or half-bigon $D'$ and set $C' = C$. We then perform the isotopy of $C'$ across the maximal enlargement of the isotopy region associated with $D'$. If this reduces the bigon number, then we return to the start of the loop. On the other hand, if the bigon number is unchanged, then we know, by Lemma \ref{Lem:BigonDoesNotGoUp}, that the new 1-manifold $C''$ has an arc of $C'' \cut P$ that runs parallel to $D' \cap C'$. An outermost such arc bounds a bigon or half-bigon. We repeat the above procedure with this disc. We continue until the bigon number decreases, at which point we return to the start of the loop. This continues until there are no more bigons or half-bigons.

The method of finding maximal Reeb and semi-Reeb regions and performing the isotopies is exactly analogous to that in the proof of Theorem \ref{Thm:Normalise1Manifold}. The argument that this procedure works is also similar, using Lemmas \ref{Lem:NotSelfReturningBigon} and \ref{Lem:SelfReturningHorizBigon} and analogues of Lemmas \ref{Lem:SelfReturningNotReeb}, \ref{Lem:ReebConclusion} and \ref{Lem:ReebIsotopy}. Note that the upper bounds on the number of isotopies in these lemmas depend on $|\partial P|$. But when $P$ is as in (2) in Theorem \ref{Thm:MinimalPosition}, $|\partial P|$ is at most $2||\calH||$ or $3|\calT|$.

The output from the above procedure is a standard curve $C''$ that is in minimal position with respect to $P$, but that is not necessarily normal. 

By Lemma \ref{Lem:IncreaseInSimp}, each of the isotopies that we have used increases the simplification number by at most $4||\calH||$. Hence, $\mathrm{simp}(C'')$ is bounded above by a polynomial function of $\mathrm{simp}(C)$ and $||\calH||$. 

We also need to control the weight of $C''$. This 1-manifold is produced from $C$ in steps, where each step is a Reeb isotopy, a semi-Reeb isotopy or an isotopy across the maximal enlargement of an isotopy region. Let $C_i$ be the 1-manifold obtained after the $i$th step. We claim that the weight of $C_i$ is at most $w(C) + 2(|C \cap P| - |C_i \cap P|)|\calT|$ or $w(C) + 2(|C \cap P| - |C_i \cap P|)||\calH||$. When $i = 0$, this is trivially true. When we have a handle structure $\calH$, then in going from $C_i$ to $C_{i+1}$, the weight increases by at most $2 ||\calH||(|C_i \cap P| - |C_{i+1} \cap P|)$, by Lemma \ref{Lem:IncreaseInSimp}. So, we have
\begin{align*}
w(C_{i+1}) &\leq w(C_i) + 2 ||\calH||(|C_i \cap P| - |C_{i+1} \cap P|) \\
&\leq w(C) + 2(|C \cap P| - |C_i \cap P|)||\calH|| + 2 ||\calH||(|C_i \cap P| - |C_{i+1} \cap P|) \\
&= w(C) + 2(|C \cap P| - |C_{i+1} \cap P|)||\calH||.
\end{align*}
This establishes the claim. Thus, $w(C'') \leq w(C) + 2|C \cap P| ||\calH|| \leq w(C)(1 + 2||\calH||)$. A similar statement is true when $S$ has a triangulation.

We can therefore apply {\sc Normalise a standard 1-manifold in a handle structure} (Theorem \ref{Thm:Normalise1Manifold}) or {\sc Normalise standard 1-manifold in a triangulation} (Theorem \ref{Thm:NormaliseStandardCurves}) to normalise $C''$ in polynomial time. Note that by Corollary \ref{Cor:RemainsMinimalPosition} below, the resulting 1-manifold $C'''$ remains in minimal position with respect to $P$. Furthermore, $w(C''') \leq w(C'')$.
\end{proof}

\begin{lemma}
\label{Lem:StaysMinimal}
Let $C$ be a standard 1-manifold properly embedded in a compact surface $S$ with a handle structure $\calH$. Let $P$ be a properly embedded 1-manifold that is cellular with respect to the cell structure dual to $\calH$. Suppose that $C$ and $P$ are in minimal position.
Let $C'$ be obtained from $C$ by isotopy across a simplifying disc and the incident 1-handle. Then $C'$ and $P$ are in minimal position.
\end{lemma}

\begin{proof}
The simplifying disc does not extend to a bigon or half-bigon of $C$ and $P$, by the assumption that they are in minimal position. Hence, the isotopy across the simplifying disc is an isotopy supported in the complement of $P$. Hence, the resulting 1-manifold $C'$ remains in minimal position with respect to $P$.
\end{proof}

\begin{corollary}
\label{Cor:RemainsMinimalPosition}
Let $S$, $C$ and $\calH$ or $\calT$ be as in Theorem \ref{Thm:Normalise1Manifold} or Theorem \ref{Thm:NormaliseStandardCurves}. Let $P$ be a 1-manifold properly embedded in $S$ that is either cellular with respect to the cell structure dual to $\calH$ or a subset of the 1-skeleton of $\calT$. Suppose that $C$ and $P$ are in minimal position. Then the output $C'$ from the algorithm {\sc Normalise a standard 1-manifold in a handle structure} or {\sc Normalise standard 1-manifold in a triangulation} is also in minimal position with respect to $P$.
\end{corollary}

\begin{proof} As discussed in Remark \ref{Rem:CompositionSmallIsotopies}, the isotopies used in {\sc Normalise a standard 1-manifold in a handle structure} or {\sc Normalise standard 1-manifold in a triangulation} can be viewed as a composition of isotopies, each of which slides the 1-manifold across a simplifying disc and an incident 1-handle. By Lemma \ref{Lem:StaysMinimal}, each of these isotopies keeps the 1-manifold and $P$ in minimal position.
\end{proof}

\section{Geometric intersection number and the isotopy problem}
\label{Sec:GeometricIntersectionNumber}

In this section, we complete the proof of Theorems \ref{Thm:GeometricIntersectionNumber} and \ref{Thm:IsotopyProblem}.

\begin{named}{Theorem \ref{Thm:GeometricIntersectionNumber}}[{\sc Geometric intersection number}]
Let $\mathcal{T}$ be a triangulation for a compact orientable surface $S$. Let $C_1$ and $C_2$ be normal 1-manifolds properly embedded in $S$. Then there is an algorithm that provides the geometric intersection number of $C_1$ and $C_2$. The running time of this algorithm is bounded above by a polynomial function of $|\mathcal{T}|$, $\log w(C_1)$ and $\log w(C_2)$.
\end{named}

\begin{named}{Theorem \ref{Thm:IsotopyProblem}}[{\sc Isotopy of 1-manifolds}]
Let $\mathcal{T}$ be a triangulation for a compact orientable surface $S$. Let $C_1$ and $C_2$ be normal essential 1-manifolds properly embedded in $S$. Then there is an algorithm that determines whether $C_1$ and $C_2$ are isotopic. The running time of this algorithm is bounded above by a polynomial function of $|\mathcal{T}|$, $\log w(C_1)$ and $\log w(C_2)$.
\end{named}

%This should be compared with the main result of \cite{BellWebb} by Bell and Webb. They also provide an algorithm to compute the geometric intersection number of two curves $C_1$ and $C_2$. Its running time is bounded by a polynomial function of $\log w(C_1)$ and $\log w(C_2)$, but depends exponentially on the Euler characteristic of $S$. Their algorithm also does not work for closed surfaces.

%We will also prove the following analogous result for handle structures.

%\begin{theorem}[{\sc Geometric intersection number in a handle structure}]
%\label{Thm:GeometricIntersectionNumberHS}
%Let $\mathcal{H}$ be a handle structure for a compact orientable surface $S$. Let $C_1$ and $C_2$ be normal 1-manifolds properly embedded in $S$. Then there is an algorithm that provides the geometric intersection number of $C_1$ and $C_2$. The running time of this algorithm is a polynomial function of $||\mathcal{H}||$, $\log (w(C_1) + |\partial C_1|)$ and $\log (w(C_2) + |\partial C_2|)$.
%\end{theorem}

We follow the approach of Bell \cite{Bell} and Bell and Webb \cite{BellWebb}, where they modify $\calT$ to a new triangulation $\calT'$ with the property that $C_2$ intersects each edge of $\calT'$ as few times as possible. They also keep track of how $C_1$ sits in $\calT'$. 
%We will instead use handle structures, where the relevant moves are as follows.

In order to make our results as general as possible, we also consider \emph{ideal triangulations} of a compact surface $S$. Recally that this is a representation of $S - \partial S$ as a union of triangles with their edges identified in pairs and with their vertices removed.

\begin{definition}
\label{Def:Flip}
Let $\calT$ be a triangulation or ideal triangulation of a compact surface $S$. Let $e$ be an edge of $\calT$ that is adjacent to distinct triangles. Then the operation of removing $e$, forming a square, and then inserting the other diagonal of the square, creates a new triangulation or ideal triangulation that is obtained from $\calT$ by a \emph{flip}. This is also called a \emph{2-2 Pachner move}.
\end{definition}

%\begin{definition}
%\label{Def:CollapseExpand}
%Let $\calH$ be a handle structure for a compact surface.
%\begin{enumerate}
%\item Let $H_0$ and $H_1$ be a 0-handle and a 1-handle with non-empty connected intersection. Let $H'_0$ be the other 0-handle incident to $H_1$. If we replace $H_0 \cup H_1 \cup H'_0$ with a single 0-handle, the resulting handle structure $\calH'$ is obtained from $\calH$ by a \emph{collapse}. 
%\item Let $H_1$ and $H_2$ be a 1-handle and a 2-handle with non-empty connected intersection. If $H_1$ is incident to another handle $H'_2$, then replacing $H_2 \cup H_1 \cup H_2'$ by a single 2-handle results in a handle structure $\calH'$ that is obtained from $\calH$ by a \emph{collapse}. If $H_1$ is not incident to another 2-handle, then removing $H_1 \cup H_2$ is also a \emph{collapse}, giving a handle structure $\calH'$.
%\end{enumerate}
%We also say that $\calH$ is obtained from $\calH'$ by an \emph{expansion}.
%\end{definition}

%\begin{definition}
%Let $\calH$ be a handle structure of a compact surface $S$. Let $C$ be a simple closed curve that is cellular with respect to the dual cell structure. Then, for a natural number $n$, a \emph{twist of order $n$} about $C$ is the modification to $\calH$ that results from Dehn twisting $n$ times about $C$.
%\end{definition}

\begin{definition}
Let $\calT$ be a triangulation or ideal triangulation of a compact surface $S$. Let $C$ be a normal simple closed curve that intersects each edge of $\calT$ at most three times. Then, for a natural number $n$, a \emph{twist of order $n$} about $C$ is the modification to $\calT$ that results from Dehn twisting $n$ times about $C$.
\end{definition}

The following result is essentially a restatement of \cite[Theorems 7.4 and 7.5]{LackenbyYazdi}. A similar result was proved by Bell and Webb \cite[Corollary 3.5]{BellWebb}, but they did not control the running time in terms of the Euler characteristic of the surface.

\begin{theorem}[{\sc Simplify normal 1-manifold}]
\label{Thm:SimplifyNormal1Manifold}
There is an algorithm that takes as its input the following:
\begin{enumerate}
\item a 1-vertex triangulation $\mathcal{T}$ of a closed orientable surface $S$ or an ideal triangulation $\calT$ of a compact orientable surface $S$ with non-empty boundary,
\item a properly embedded connected essential normal 1-manifold $C$, 
\end{enumerate}
and outputs a sequence of 1-vertex triangulations or ideal triangulations 
$$\calT = \calT_0, \calT_1, \dots, \calT_k$$
and, for each $i$, a properly embedded 1-manifold $C_i$ that is normal in the handle structure dual to $\calT_i$, such that:
\begin{enumerate}
\item each $\calT_i$ is obtained from $\calT_{i-1}$ by a flip or a twist of order $n_i$, where $n_i$ is bounded above by a polynomial function of $w(C)$ and $|\calT|$;
\item $C_i$ is isotopic to $C_{i-1}$;
\item each component of $C_k$ intersects each edge of $\calT_k$ at most one twice;
\item the number of moves $k$ is bounded above by a polynomial function of $|\calT|$ and  $\log w(C)$.
\end{enumerate}
The running time of the algorithm is bounded above a polynomial function of $|\calT|$ and $\log w(C)$.
\end{theorem}

\begin{remark}
In Theorems 7.4 and 7.5 of \cite{LackenbyYazdi}, the surface $S$ was required to be orientable, and so the same restriction is made here and hence in Theorems 
\ref{Thm:GeometricIntersectionNumber} and \ref{Thm:IsotopyProblem}. The requirement that $S$ was orientable was necessary for the proofs in \cite{LackenbyYazdi}, which relied on an analysis of the AHT algorithm applied to train tracks, and certain operations in the AHT algorithm can be ruled out under the assumption that $S$ is orientable. Nevertheless, we expect that Theorem \ref{Thm:SimplifyNormal1Manifold} remains true even when $S$ is non-orientable, and hence we expect that the orientability hypothesis can be dropped from Theorems \ref{Thm:GeometricIntersectionNumber} and \ref{Thm:IsotopyProblem}.
\end{remark}

\begin{proof}[Proof of Theorem \ref{Thm:GeometricIntersectionNumber}] 
We are given a triangulation $\calT$ for a compact orientable surface $S$, and normal co-ordinates for two properly embedded 1-manifolds $C_1$ and $C_2$. We wish to compute the geometric intersection number $i(C_1, C_2)$.

We may assume that $C_2$ is connected, for the following reason. Using {\sc Components of standard 1-manifold} (Theorem \ref{Thm:ComponentsCurves}), we may compute a list of components of $C_2$. It is easy to see that $i(C_1, C_2)$ is equal to the sum of $i(C_1, C_2')$ over all components $C'_2$ of $C_2$. Hence, if we can compute $i(C_1, C_2')$ for each such component, then we can compute $i(C_1, C_2)$.

We may assume that $S$ is not a sphere, disc or annulus, since geometric intersection number is straightforward to compute in these cases. Hence, 
we may reduce $\calT$ to a 1-vertex triangulation or ideal triangulation using the algorithm in Proposition 10.3 in \cite{Lackenby:Efficient} for example. We sketch the procedure now.

First suppose that $S$ is closed. We will show how to reduce the number of vertices of $\calT$ to $1$. If there is an edge with the same triangle on both sides, then one endpoint of this edge is a degree one vertex $v$. The remaining edge has distinct triangles on each side. Hence, a flip may be performed along the latter edge, which increases the degree of $v$ to $2$. Then a further flip may be performed, increasing the degree of $v$ to $3$. A 3-1 Pachner move may then be performed, which removes the three triangles around $v$ and replaces them with a single triangle. This reduces the number of vertices. Hence, we may assume that each edge has distinct triangles on each side. Now in any triangulation of the closed surface $S$, the average vertex degree is at most $\max\{ 6, 6 - 6\chi(S)\}$, and so there is always a vertex $v$ with degree at most this quantity. If $v$ is not the unique vertex of the triangulation, there is an edge emanating from $v$ that ends on a distinct vertex. If every edge emanating from $v$ has this form, then we can perform a flip along any of these edges to reduce the degree of $v$ by $1$. On the other hand, if there is some edge emanating from $v$ that ends on $v$, then we can find a triangle with two vertices equal to $v$ and one vertex not equal to $v$. If we perform a flip along the $v$-$v$ edge, then we reduce the degree of $v$ by at least $1$. So we can reduce the degree of $v$ to at most $3$, and it can then be removed. Thus, we eventually obtain a 1-vertex triangulation. The number of moves we used was at most a polynomial function of $\chi(S)$ and $|\calT|$.

A similar procedure can be used when $S$ has non-empty boundary. We can attach triangulated discs to $\partial S$ to form a triangulation of a closed surface $\overline{S}$, thereby increasing the number of triangles by at most a factor of $2$. Pick a vertex in each of these discs. We can then remove all remaining vertices using flips and 3-1 Pachner moves, just as in the above procedure. The result is a triangulation of $\overline{S}$ with one vertex in each of the new discs and no others. Now remove these vertices to form an ideal triangulation of $S$. Again, the number of moves we used was at most a polynomial function of $\chi(S)$ and $|\calT|$.

During this process, we can keep track of the 1-manifolds $C_1$ and $C_2$. In the case of flips, this is straightforward. But in the case of a 3-1 Pachner move, non-normal arcs of $C_1$ or $C_2$ may be created in the new triangle. However, these can easily be normalised using {\sc Normalise standard 1-manifold in a triangulation}. 
Each flip increases $w(C_1)$ and $w(C_2)$ by at most a factor of $2$. Hence, $\log(w(C_1))$ and $\log(w(C_2))$ increase by at most a polynomial function of $\chi(S)$ and $|\calT|$.

Denote this new 1-vertex triangulation or ideal triangulation by $\calT'$.

We now apply {\sc Simplify normal 1-manifold} from Theorem \ref{Thm:SimplifyNormal1Manifold} to $\calT'$ and the 1-manifold $C_2$. This converts $\calT'$ to a new 1-vertex or ideal triangulation $\calT''$, and it converts $C_2$ to a properly embedded normal 1-manifold that intersects each edge of $\calT''$ at most twice. (Here, we have used that $C_2$ is connected.) This triangulation $\calT''$ is obtained from $\calT'$ by a sequence of flips and twists. At each step, we can keep track of the 1-manifold $C_1$, until it is a normal 1-manifold in $\calT''$. We can subdivide $\calT''$ to form a triangulation $\calT'''$ of $S$ that contains $C_2$ as a subcomplex, and where $|\calT'''| \leq 9 |\calT''|$. 
We can also compute $C_1$ as a normal curve in this refinement. We apply {\sc Minimal position for 1-Manifolds} from Theorem \ref{Thm:MinimalPosition}, giving a 1-manifold $C_1'$ that is isotopic to $C_1$ and in minimal position with respect to $C_2$. Their number of intersection points is the required output from our algorithm.
\end{proof}

\begin{proof}[Proof of Theorem \ref{Thm:IsotopyProblem}]
Again we are given a triangulation $\calT$ for a compact orientable surface $S$, and normal co-ordinates for two properly embedded 1-manifolds $C_1$ and $C_2$. We wish to compute whether they are isotopic.

Again we may assume that $C_1$ and $C_2$ are connected, for the following reason. Using {\sc Components of standard 1-manifold} (Theorem \ref{Thm:ComponentsCurves}), we may compute a list of components of $C_1$ and $C_2$. These fall into at most $O(|\chi(S)|)$ topological parallelism classes. We can determine these parallelism classes, since two components $C'_1$ and $C''_1$ of $C_1$, say, are topologically parallel if and only if they cobound an annulus or square. We can determine this by using {\sc Cut along standard 1-manifold} (Theorem \ref{Thm:CutAlong1Manifold}) to compute the reduced handle structure on $S \cut (C'_1 \cup C''_1)$ and then determining the topological type of each component of this surface, as well as the information about which parts of the boundary came from $C'_1$ and $C''_1$. So we now take one representative $C'_1$ from a topological parallelism class of the components of $C_1$. We also run through each topological parallelism class of components $C'_2$ of $C_2$ in turn. Suppose that we can determine whether $C'_1$ and $C'_2$ are isotopic. If $C'_1$ is not isotopic to any component of $C_2$, then the algorithm terminates with a negative answer. On the other hand, if $C'_1$ is isotopic to some component $C'_2$ of $C_2$, then we determine the number of parallel copies of these components in $C_1$ and $C_2$. If these numbers are different, we again terminate with a negative answer. On the other hand, if there are the same number of components, then we remove these components from $C_1$ and $C_2$ and repeat.

So assume that $C_1$ and $C_2$ are connected.
As above, we modify $\mathcal{T}$ to a 1-vertex triangulation or ideal triangulation $\calT'$. We then apply {\sc Simplify normal 1-manifold} from Theorem \ref{Thm:SimplifyNormal1Manifold} to $\calT'$ and the 1-manifold $C_2$, and then we subdivide to give a triangulation $\calT'''$ in which $C_2$ is simplicial. We also keep track of an isotopic copy of $C_1$ in $\calT'''$. We apply {\sc Minimal position for 1-Manifolds} from Theorem \ref{Thm:MinimalPosition}, giving a 1-manifold (also called $C_1$) that is isotopic to $C_1$ and in minimal position with respect to $C_2$. If $C_1$ and $C_2$ intersect in minimal position, they are not isotopic. On the other hand, if they do not intersect, then we can use {\sc Cut along standard 1-manifold} to compute the reduced handle structure on $S \cut (C_1 \cup C_2)$ and then determine whether any component is an annulus or square interpolating between $C_1$ and $C_2$. There is such a component if and only if $C_1$ and $C_2$ are isotopic.
\end{proof}

\section{Surfaces containing a pattern}
\label{Sec:Pattern}

In this section, we will consider a closed surface $S$ containing a pattern $P$, which is defined as follows.

\begin{definition}
A \emph{pattern} $P$ in a closed surface $S$ is a disjoint union of simple closed curves and embedded trivalent graphs. 
\end{definition}

The reason for considering such a situation is that the 3-manifolds appearing in a hierarchy naturally inherit a pattern in their boundary \cite{HakenHomeomorphism, Matveev}. The analysis that we present here will be useful in future work of the author on hierarchies.

%We now wish to consider a related problem, which can be motivated by the following question. Recall that two simple closed curves $C_1$ and $C_2$ in a surface $S$ are in \emph{minimal position} if there is no isotopy that reduces the number of intersection points of $C_1$ and $C_2$. There is a unique position, up to isotopy, for the union of two such curves (provided they are both essential). The number of intersection points of $C_1$ and $C_2$ in minimal position is their \emph{geometric intersection number}. Can we efficiently compute the geometric intersection number of two normal curves in a surface? There is an algorithm given by Bell and Webb \cite{BellWebb} that works for surfaces $S$ with at least one boundary component, and where $C_1$ and $C_2$ are given as normal curves with respect to some ideal triangulation of $S$. The running time is bounded above by a polynomial function of $\log w(C_1)$ and $\log w(C_2)$. However, its running time increases rapidly as one varies $\chi(S)$. In the next subsection, we present an algorithm that works also for closed surfaces and with running time that is bounded by a polynomial function of the genus of $S$.

%In this section, we first consider how to arrange $C_1$ and $C_2$ to be in minimal position. We will consider an asymmetric set-up, where $C_1$ is assumed to have simple intersection with a given handle structure. Indeed, we will also consider a slightly more general situation where some components of $C_1$ are trivalent graphs. We therefore use the symbol $P$ for $C_1$.

We now generalise many of the definitions from earlier in the paper to the setting of patterns.

We will consider 1-manifolds $C$ properly embedded in a closed surface $S$ with a pattern $P$. After a small isotopy of $C$, we may assume that $C$ is disjoint from the vertices of $P$ and intersects its edges and simple closed curves transversely. Henceforth, we will always assume this to be the case.

\begin{definition}
Let $C$ be a properly embedded 1-manifold in the closed surface $S$, and let $P$ be a pattern in $S$. Then a \emph{bigon disc} is a disc $D$ embedded in $S$ such that $\partial D$ is the concatenation of arcs $\alpha$ and $\beta$ where $\alpha = D \cap C$ and $\beta = D \cap P$ is disjoint from the vertices of $P$. 
%A \emph{half-bigon disc} is a disc $D$ embedded in $S$ such that $\partial D$ is the concatenation of three arcs $\alpha_1$, $\alpha_2$ and $\alpha_3$ where $\alpha_1 = D \cap P$ is disjoint from the vertices of $P$, and $\alpha_2 = D \cap C$ and $\alpha_3 = D \cap \partial S$. 
A \emph{trigon disc} is a disc $D$ embedded in $S$ such that $\partial D$ is the concatenation of arcs $\alpha$ and $\beta$, where $\alpha = D \cap C$ and $\beta = D \cap P$ intersects a single vertex of $P$ in its interior. (See Figure \ref{Fig:IsotopyRegionTrigon}.)
\end{definition}

\begin{definition}
We say that $C$ and $P$ are in \emph{locally minimal position} if they have no bigon or trigon disc.
\end{definition}

%The reason we consider the more general situation is as follows. Suppose that $M$ is a compact orientable 3-manifold with a boundary pattern $P$. Let $F$ be a properly embedded surface that is in general position with respect to $P$. Then if $\partial F$ and $P$ are not in minimal position, then $F$ admits a pattern-compression disc, unless $F$ is a disc and $\partial F$ also bounds a disc in $\partial M$ intersecting $P$ either in an arc or a tripod. For suppose that $D$ is a bigon or trigon disc. Let $N(D)$ be a regular neighbourhood of $D$ in $M \cut F$. Let $D'$ be the disc $\partial N(D) \cut (F \cup \partial M)$. Then $D'$ is a pattern-compression disc.

%\begin{definition}
%Let $C_1$ and $C_2$ be properly embedded 1-manifolds in the compact orientable surface $F$. Then a \emph{reduction disc} is a disc $D$ embedded in $F$ satisfying one of the following:
%\begin{enumerate}
%\item $\partial D$ is the concatenation of arcs $\alpha_1$ and $\alpha_2$ where $\alpha_1 = D \cap C_1$ and $\alpha_2 = D \cap C_2$; this is a \emph{bigon}
%\item $\partial D$ is the concatenation of arcs $\alpha_1$, $\alpha_2$ and $\beta$, where $\alpha_1 = D \cap C_1$, $\alpha_2 = D \cap C_2$ and $\beta = D \cap \partial F$.
%\end{enumerate}
%\end{definition}

By analogy with the simplification number and bigon number, we present the following definition.

\begin{definition}
The \emph{bi-trigon number} $\text{bi-trigon}(C, P)$ of $C$ and $P$ is the number of trigon discs and plus twice the number of bigon discs.
\end{definition}

We now consider what the right analogue of an isotopy region should be. 

\begin{definition}
\label{Def:IsotopyRegionTrigon}
Let $D$ be a bigon or trigon for $C$ and $P$. Then the \emph{associated isotopy region} is a disc $R$ embedded in $S$ that is the union of $D$ and a subset of the form $I \times [-1,1]$ for a closed bounded interval $I$, with the following properties:
\begin{enumerate}
\item the intersection between $D$ and $I \times [-1,1]$ is the arc $\partial D \cut C$; it is also a component of $\partial I \times [-1,1]$;
\item the intersection between $C$ and $I \times [-1,1]$ is $I \times \{ -1,1 \}$;
\item the intersection between the vertices of $P$ and $I \times [-1,1]$ lies on $I \times \{ 0 \}$;
\item each component of intersection between $I \times [-1,1]$ and an edge of $P$ is of the form $I' \times \{ 0 \}$ for a closed bounded interval $I' \subseteq I$, or $\{ \ast \} \times [-1,0] $, or $\{ \ast \} \times [0,1] $ or $\{ \ast \}  \times [-1,1]$, for a point $\ast \in I$, or possibly a union of such arcs.
\end{enumerate}
Furthermore, we require that $I \times [-1,1]$ is maximal with these properties, up to isotopy preserving $P$ and $C$. (See Figure \ref{Fig:IsotopyRegionTrigon}.) Let $C'$ be the 1-manifold obtained from $C$ by isotoping $C$ across $R$.
We say that $D$ and $R$ are \emph{self-returning} if the component of $C' \cut P$ containing $\partial R \cut C$ is topologically parallel to $D \cap C$ and all the arcs of $C \cut P$ between them are also parallel to $D \cap C$.
\end{definition}

\begin{figure}[h]
\centering
\includegraphics[width=0.6\textwidth]{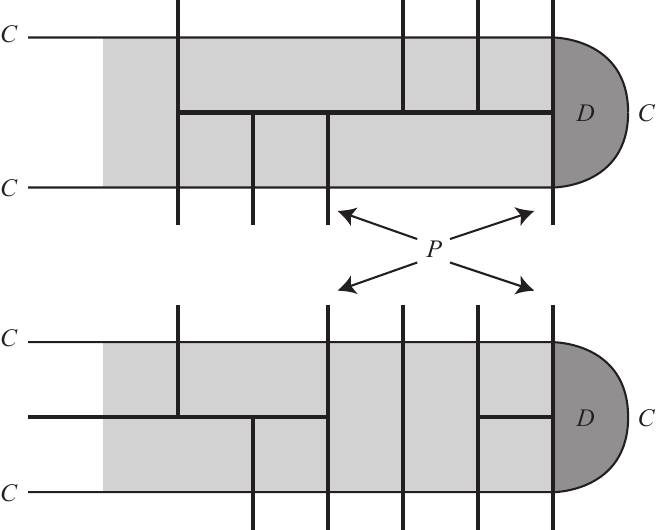}
\caption{Two possibilities for the isotopy region associated with the trigon $D$.}
\label{Fig:IsotopyRegionTrigon}
\end{figure}

The main property of this definition is that isotopy regions are extendable, in the following sense. Suppose that $R$ is a region satisfying the above properties, but without the condition of maximality. Suppose that we were to isotope $C$ across $R$ creating a 1-manifold $C'$. Suppose that $C'$ has a bigon or trigon disc $D'$ that intersects $C' \cut C$ and with interior disjoint from $R$ and $C$. Then a small regular neighbourhood of $R \cup D'$ in $S \cut C$ would also satisfy the definition above (possibly without maximality). (See Figure \ref{Fig:ExtendTrigonIsotopyExtend}.) Hence, we deduce that when $R$ is indeed maximal, then there is no such bigon or trigon disc $D'$. This is a key step in the proof of the following result, which is an analogue of Lemmas \ref{Lem:SimpDoesNotGoUp} and \ref{Lem:BigonDoesNotGoUp}.

\begin{figure}[h]
\centering
\includegraphics[width=0.65\textwidth]{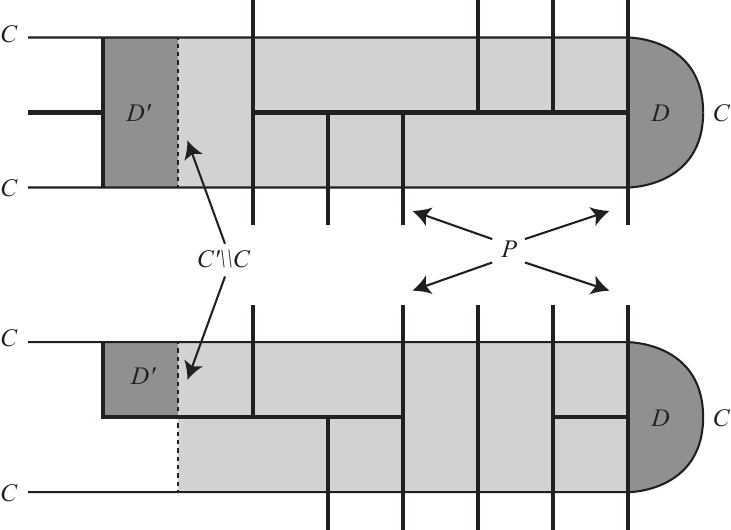}
\caption{Shown are two regions $R$ that satisfy the conditions of an isotopy region, but possibly without maximality. Shown also is $C' \cut C$, where $C'$ is obtained by isotoping $C$ across $R$. A bigon or trigon $D'$ for $C'$ is also shown. In each case, a small regular neighbourhood of $R \cup D'$ in $S \cut C$ also satisfies the definition of an isotopy region, possibly without maximality.}
\label{Fig:ExtendTrigonIsotopyExtend}
\end{figure}

\begin{lemma}
\label{Lem:BiTrigonDoesNotGoUp}
Let $C$ be a 1-manifold properly embedded in a closed surface $S$ containing a pattern $P$.
Let $D$ be a bigon or trigon for $C$ and $P$.
Let $C'$ be the 1-manifold obtained by isotoping $C$ across the isotopy region associated with $D$. Then $\mathrm{bi}\text{-}\mathrm{trigon}(C', P) \leq \mathrm{bi}\text{-}\mathrm{trigon} (C,P)$ and $|C' \cap P| < |C \cap P|$. Furthermore, if $\mathrm{bi}\text{-}\mathrm{trigon}(C',P) =  \mathrm{bi}\text{-}\mathrm{trigon}(C,P)$, then there is an arc of $C' \cut P$ that is topologically parallel to $D \cap C$, and all the components of $C \cut P$ between these two arcs are topologically parallel to $D \cap C$.
\end{lemma}

\begin{proof}
Just as in the proof of Lemma \ref{Lem:SimpDoesNotGoUp}, we define a function 
$$f \colon \{ \text{bigons and trigons for } C' \} \rightarrow \mathbb{P} \{ \text{bigons and trigons for } C \},$$
by declaring that for each bigon or trigon $D'$ for $C'$, $f(D')$ is the set of bigon and trigon discs for $C$ that lie in $D'$. There are obvious analogues of Properties (1) - (4) in the proof of Lemma \ref{Lem:SimpDoesNotGoUp} which also hold here:
\begin{enumerate}
\item $f(D') \not= \emptyset$ for every bigon or trigon disc $D'$ for $C'$;
\item $f(D'_1) \cap f(D'_2) = \emptyset$ for distinct bigon or trigon discs $D'_1$ and $D'_2$ for $C'$;
\item if $D'$ is a bigon disc for $C'$, then $f(D')$ consists of bigon discs;
\item if $f(D') = \{ D_1 \}$ for a single disc $D_1$ that is a bigon exactly when $D'$ is a bigon, then $C\cap \mathrm{int}(D')$ consists of a (possibly empty) collection of parallel arcs, all parallel to $D' \cap C'$.
\end{enumerate}
They quickly imply the lemma. 

The only property that is not immediate is (1): $f(D') \not= \emptyset$ for every bigon or trigon disc $D'$ for $C'$. We now sketch how this is proved.
For a bigon or trigon disc $D'$ for $C'$ that is also a bigon or trigon disc for $C$, this is immediate. So suppose that $D'$ is a bigon or trigon for $C'$ that is not a bigon or trigon for $C$. Suppose also, for contradiction, that $D'$ does not contain a bigon or trigon of $C$. Let $\alpha'$ be the arc $\partial D' \cap C'$. If $\alpha'$ is a subset of $C$, then it must have been the case that the interior of $D'$ contained an arc of intersection with $C$, and an outermost such arc gives a bigon or trigon for $C$, which we are assuming is not the case. So suppose that $\alpha'$ is not a subset of $C$. It must then be incident to the arc $\partial R \cut C$. If the interior of $R$ is disjoint from the interior of $D'$, then, as observed above, a small regular neighbourhood of $R \cup D'$ in $S \cut C$ would also satisfy the definition of an isotopy region for $C$ (possibly apart from maximality), contradicting the fact that $R$ is maximal. Hence, $R$ must enter the interior of $D'$ at the arc $\partial R \cut C$. But this then implies that $D'$ contains a bigon or trigon of $C$. \end{proof}

\begin{definition}
Let $D$ be a bigon or trigon disc in the closed surface $S$, and let $R$ be the associated isotopy region. Then an \emph{enlargement} of $D$ is a disc $D_+$ in $S$ such that 
\begin{enumerate}
\item $D_+$ contains $R$;
\item $D_+ \cap C$ consists of arcs topologically parallel to $R \cap C$;
\item one of these arcs is $\partial D_+ \cap C$;
\item the remainder of $\partial D_+$ contains $\partial R \cut C$ and has no points of intersection with $P$ other than $(\partial R \cut C) \cap P$.
\end{enumerate}
%\item $D_+ \cap C$ is an arc in the boundary of $\partial D_+$;
%\item the remainder of $\partial D_+$ is an arc in $P$;
%\item $D_+ \cap C_2$ consists of arcs parallel $\partial D_+ \cap C_2$;
%\item $D_+$ contains $D$.
%\end{enumerate}
The \emph{maximal enlargement} of $D$ is an enlargment that is not contained within a bigger enlargement, up to ambient isotopy preserving $P$ and $C$. The \emph{associated isotopy} of $C$ removes the arcs $D_+ \cap C$ and replaces them with parallel copies of $\partial R \cut C$.
\end{definition}

The following is an analogue of Lemma \ref{Lem:MaxEnlargementIsotopySimp}.

\begin{lemma}
\label{Lem:MaxEnlargementIsotopyTrigon}
Let $D$ be a bigon or trigon for $C$ and $P$. Let $C'$ be the 1-manifold that is obtained from $C$ by an isotopy across a maximal enlargement of the isotopy region associated with $D$. Then $\mathrm{bi}\text{-}\mathrm{trigon}(C', P) \leq \mathrm{bi}\text{-}\mathrm{trigon}(C, P)$.
\end{lemma} 

The following is analogue of Lemma \ref{Lem:NotSelfReturning}, with essentially the same proof.

\begin{lemma}
\label{Lem:TrigonNotSelfReturning}
Let $D$ be a bigon or trigon for $C$ and $P$ that is not self-returning. Let $e(P)$ be the number of edges of $P$. Then after performing $O(|\chi(S)| + e(P) + \mathrm{bi}\text{-}\mathrm{trigon}(C, P))$ isotopies across maximal enlargements of isotopy regions containing $D$, the result is a 1-manifold $C'$ with $\mathrm{bi}\text{-}\mathrm{trigon}(C', P) < \mathrm{bi}\text{-}\mathrm{trigon}(C, P)$.
\end{lemma}  

There are also analogues of Lemmas \ref{Lem:SelfReturningHoriz} and \ref {Lem:SelfReturningNotReeb}. In the case where a trigon or bigon disc $D$ is self-returning and where its isotopy region $R$ has no arc of $(R \cap C) \cut P$ parallel to $D \cap C$, we have the following analogue of a Reeb region.

\begin{definition}
Let $S$ be a closed surface containing a pattern $P$. Let $C$ be a 1-manifold properly embedded in $S$. A \emph{generalised Reeb region} is a subsurface $A$ of $S$ homeomorphic to an annulus or M\"obius band  $S^1 \widetilde{\times} [-1,1]$ and with the following
properties:
\begin{enumerate}
\item the intersection between the vertices of $P$ and $S^1 \widetilde{\times} [-1,1]$ lies on $S^1 \widetilde{\times} \{ 0 \}$;
\item each component of intersection between $S^1 \widetilde{\times} [-1,1]$ and an edge of $P$ is of the form $I' \widetilde{\times} \{ 0 \}$ for a closed bounded interval $I' \subseteq S^1$, or $\{ \ast \} \widetilde{\times} [-1,0] $, or $\{ \ast \} \widetilde{\times} [0,1] $ or $\{ \ast \}  \widetilde{\times} [-1,1]$, for a point $\ast \in S^1$, or possibly a union of such arcs.
\item $A$ contains a bigon or trigon $D$ in its interior;
\item the intersection between $A$ and $C$ consists of a union of essential arcs in $A$, all of which are topologically parallel;
\item each such arcs is a concatenation of three arcs $\alpha_-, \alpha_0, \alpha_+$;
\item $\alpha_-$ starts at a point of $\partial A$ and then runs parallel to $\partial A$ in some direction, possibly winding several times around that component but winding at least once around that component;
\item $\alpha_0$ is topologically parallel to $D \cap C$;
\item $\alpha_+$ starts at a point of $\partial A$ opposite $\partial A \cap \alpha_-$, and winds around $\partial A$ the same number of times as $\alpha_-$ and in the same direction.
\end{enumerate}
A \emph{generalised Reeb isotopy} is supported in a generalised Reeb region as above. It replaces $A \cap C$ with essential arcs in $A$ of the form $\{ \ast \} \widetilde{\times} [-1,1]$ for a point $\ast \in S^1$. A Reeb region $A$ is \emph{maximal} if any other Reeb region containing it is isotopic to $A$ via an isotopy preserving $C$ and $P$ throughout. (See Figure \ref{Fig:GenerlisedIsotopyRegionTrigon}.)
\end{definition}

\begin{figure}[h]
\centering
\includegraphics[width=0.6\textwidth]{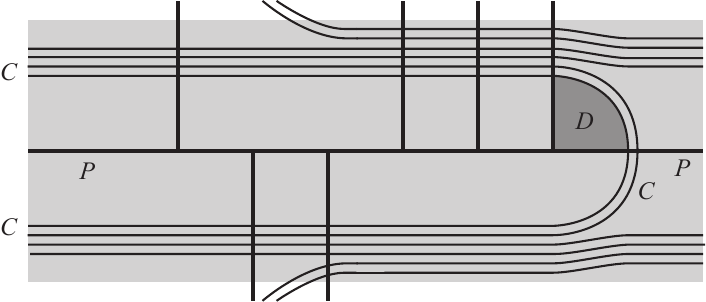}
\caption{A generalised Reeb region associated with a trigon $D$.}
\label{Fig:GenerlisedIsotopyRegionTrigon}
\end{figure}

\begin{theorem}[{\sc Locally minimal position for 1-manifold and pattern}]
\label{Thm:MinimalPositionPattern}
There is an algorithm that takes, as its input, the following:
\begin{enumerate}
\item a triangulation $\calT$ or handle structure $\calH$ of a closed surface $S$;
\item a pattern $P$ properly embedded in $S$ that either is a subset of the 1-skeleton of $\calT$ or that intersects $\calH$ in the following way: it is disjoint from 0-handles; it intersects each 1-handle in the empty set or in the co-core of the handle; it intersects each 2-handle in the empty set, or in a properly embedded arc, or in a graph that consists of a single vertex of $P$ with three arcs running to the boundary of the 2-handle;
\item a standard 1-manifold $C$;
\end{enumerate}
and provides as its output a normal 1-manifold $\overline{C}$ that is obtained from $C$ by an isotopy, plus possibly removing some components that bound discs. This 1-manifold $\overline{C}$ is in locally minimal position with respect to $P$ and has weight that is bounded above by $w(C)(1 + 2 |\calT|)$ or $w(C)(1 + 2 ||\calH||)$. The running time is bounded above by a polynomial function of $|\calT|$ or $||\calH||$, $\mathrm{simp}(C)$ and $\log w(C)$.
\end{theorem}

\begin{proof}
This follows the same format as the proof of Theorems \ref{Thm:Normalise1Manifold} and \ref{Thm:MinimalPosition}. 

First we search for a generalised Reeb region, arising from a bigon or trigon, and we perform the associated isotopy, creating a 1-manifold $C'$. If $C'$ has a bigon or trigon in the generalised Reeb region, denote it by $D'$. If there is no such disc, then the bi-trigon number has gone down, and we return to the start. If there was no generalised Reeb region, then we pick any bigon or trigon $D'$ and set $C' = C$. We then perform the isotopy of $C'$ across the maximal enlargement of the isotopy region associated with $D'$. If this reduces the bi-trigon number, then we return to the start of the loop. On the other hand, if the bi-trigon number is unchanged, then the new curve $C''$ has an arc of $C'' \cut P$ that runs parallel to $D' \cap C'$. An outermost such arc bounds a bigon or trigon. We repeat the above procedure with this disc. We continue until the bi-trigon number decreases, at which point we return to the start of the loop. This continues until there are no more bigons or trigons, at which point the resulting 1-manifold $C''$ is in locally minimal position with respect to $P$.

We claim that $C''$ has weight at most $w(C) (1 + 2 |\calT|)$ or $w(C)(1 + 2 ||\calH||)$. This is proved in the same way as in the argument for Theorem \ref{Thm:MinimalPosition}. Each isotopy performed above replaces sub-arcs of $C$ with new arcs. Each of the new arcs has weight at most $2|\calT|$ or $2||\calH||$ and it intersects $P$ strictly fewer times than the original arc. The inequalities in the proof of Theorem \ref{Thm:MinimalPosition} then give that the resulting 1-manifold $C''$ has weight at most $w(C) + 2 |C \cap P| \, |\calT|$ or $w(C) + 2 |C \cap P| \, ||\calH||$, and this is at most the required bound.

Finally, we normalise $C''$ using {\sc Normalise a standard 1-manifold in a handle structure} (Theorem \ref{Thm:Normalise1Manifold}) or {\sc Normalise standard 1-manifold in a triangulation} (Theorem \ref{Thm:NormaliseStandardCurves}). This takes polynomial time, and the resulting normal 1-manifold $\overline{C}$ has weight at most that of $C''$ and is still in locally minimal position with respect to $P$.
\end{proof}

\bibliography{AlgSurfaces-biblio}
\bibliographystyle{plain}

\end{document}